\documentclass[10pt]{article}

\usepackage[margin=1in]{geometry}  
\usepackage{amsmath}               
\usepackage{amsfonts}              
\usepackage{amsthm}                
\usepackage{color}

\usepackage{setspace}

\newtheorem{theorem}{Theorem}[section]
\newtheorem{lemma}[theorem]{Lemma}
\newtheorem{proposition}[theorem]{Proposition}
\newtheorem{corollary}[theorem]{Corollary}

\newtheorem{remark}{Remark}
\newtheorem{definition}{Definition}
 
\newcommand{\ve}{\varepsilon}

\newcommand{\R}{\mathbb R}
\newcommand{\Z}{\mathbb Z}
\newcommand{\E}{\mathbf {E}}

\newcommand{\pr}{\mathbf{P}}

\DeclareMathOperator{\sign}{sign}
\renewcommand{\Re}{\operatorname{Re}}
\renewcommand{\Im}{\operatorname{Im}}

\providecommand{\keywords}[1]{\textit{Keywords:} #1}

\allowdisplaybreaks
\begin{document}

\title{Gaussian self-similar random fields with   distinct stationary properties of their rectangular increments}

\author{Vitalii  Makogin \\
Institute of Stochastics, Ulm University,  D-89069 Ulm, Germany. \\
E-mail: vitalii.makogin@uni-ulm.de
\and
 Yuliya Mishura \\
Department of Probability Theory, Statistics and Actuarial
Mathematics \\
Taras Shevchenko National University of Kyiv \\
Volodymyrska 64, Kyiv 01601, Ukraine, E-mail:myus@univ.kiev.ua
 }

\maketitle
\pagestyle {myheadings}
\markright{Gaussian self-similar random fields with   stationary rectangular increments}

\begin{abstract}
We describe two classes of Gaussian self-similar random fields: with  strictly stationary rectangular increments and with mild stationary rectangular increments. We find  explicit spectral and moving average representations for the fields with  strictly stationary rectangular increments and characterize fields with mild stationary rectangular increments   by the properties of covariance functions of their Lamperti transformations as well as  in terms of their spectral densities.  We establish  that both classes contain not only fractional Brownian sheets and we provide corresponding examples. As a by-product, we obtain a new spectral representation for the fractional Brownian motion.
\end{abstract}



\section{Introduction}\label{intro} Our paper is   devoted to self-similar Gaussian random  fields with some stationarity of rectangular increments. On the one hand, the study of self-similar random fields is pushed forward by the fact that a self-similarity arises in many natural phenomena (see for example \cite{beirme,menabde,pedrizzetti}) and on financial markets, as well as in functional limit theorems (cf. \cite{bardina,Damarackas,Pilipauskaite}) and stochastic differential equations (\cite{Xiao}). See \cite{maej}  for an overview of self-similar processes in the one-dimensional case $N = 1.$
There are several definitions of fractional Brownian fields and generalizations for self-similar property of random fields (cf. \cite{Clausel}). For example, random fields whose distributions are invariant under operator-scaling in both the time domain and the state space are presented by Bierm\'{e} {\it at el.} \cite{beirme2}. In this paper, we use Definition \ref{taq} of so-called coordinate-wise self-similarity which is formally introduced in the paper of Genton {\it et al.} \cite{taqqu}. The fractional Brownian sheet, introduced much earlier by Kamont in \cite{kamont}, became a separate object for study. The It\^{o} formula and local time for it are given by Tudor and Viens in \cite{tudor} and the spectral representation is given by  Ayache {\it et al.} in \cite{alp}. Hu {\it et al.} in the paper \cite{hu} establish  a version of the Feynman-Kac formula for the multidimensional stochastic heat equation with a multiplicative fractional Brownian sheet. For recent papers on non-Gaussian self-similar random fields we refer to \cite{Clarke} and \cite{Pilipauskaite}.

On the other hand, concerning stationarity, let us mention Yaglom who introduced and studied in \cite{yaglom} random fields with wide-sense stationary increments of the form $\Delta_{\mathbf{s}} X(\mathbf{t}) =X(\mathbf{t})-X(\mathbf{s}).$ It follows from the paper of Dobrushin \cite{dobr} that the class of Gaussian random fields with stationary increments of this form coincides with the class of Minkowski fractional Brownian fields, which are described by Molchanov and Ralchenko in the paper \cite{molch}.  Random fields with wide-sense stationary rectangular increments are characterized in the paper of Basse-O'Connor {\it et al.} \cite{basse} by their spectral representations. Puplinskait{\.e} and Surgailis studied the presence of stationary rectangular increments of limiting random fields in \cite{pupl}. In the present paper, we introduce three classes of random fields with stationary rectangular increments and give their characterization in various terms.

Let $(\Omega,\mathcal{F},\pr)$ be a probability space, large enough to contain all the objects considered below. Denote $\R^N_+=[0,+\infty)^N, N\ge 1$.
In our research we consider real-valued multiparameter stochastic processes, which are called random fields, with index set being $\R_+^N$ or $\R^N,\; N>1.$

The property of self-similarity for random fields as well as the notion of fractional random fields can be defined in several ways. We use the following definitions, where the self-similarity and fractionality   can be interpreted as coordinate-wise property.
\begin{definition}[\cite{taqqu}]
\label{taq} A real-valued random field $\{X (\mathbf{t}), \mathbf{t} \in
\R_+^N\}$ is called self-similar with index $\mathbf{H} = (H_1, \ldots
,H_N) \in (0,+\infty)^N$ if for any $(a_1, \ldots , a_N)\in
(0,+\infty)^N$
\[
\bigl\{X (a_1 t_1, \cdots, a_N t_N), \mathbf{t} \in \R_+^N \bigr\}
\stackrel{d} {=} \bigl\{{a_1^{H_1}\cdots a_N^{H_N}
X (\mathbf{t}), \mathbf{t} \in \R_+^N }\bigr\}.
\]
\end{definition}
\begin{definition}[\cite{alp}]
\label{fbs} A fractional Brownian sheet $B^\mathbf{H}=\{B^\mathbf{H}(\mathbf{t}),\mathbf{t}\in \R_+^N\}$ with Hurst index $\mathbf{H}=(H_1,\ldots,H_N)\in (0,1]^N$ is a centered Gaussian random field  with
covariance function
\[
\E\bigl(B^\mathbf{H}(\mathbf{t})B^\mathbf{H} (\mathbf{s})\bigr) =
2^{-N}\prod_{i=1}^{N}
\bigl(t_i^{2H_i}+s_i^{2H_i}-|t_i-s_i|^{2H_i}
\bigr),\quad \mathbf{t},\mathbf{s} \in \R_+^N.
\]
\end{definition}
Further, we restrict ourselves to the case $\mathbf{H}\in (0,1)^N$ because we are focus on spectral representations. In case of $N=1,$ the process above is called a fractional Brownian motion.

When the index set is multi-dimensional, we consider rectangular increments as an analogue of one-dimensional increments.  Denote by $[\mathbf{s,t}]$ the rectangle $\prod_{k=1}^N [s_k,t_k]$ for any $\mathbf{s}=(s_1,\ldots,s_N)\in \R_+^N,\mathbf{t}=(t_1,\ldots,t_N)\in \R_+^N,$ such that $s_k\leq t_k,1\leq k \leq N.$
\begin{definition}
\label{incr}
Let
 $\{X(\mathbf{t}),\mathbf{t}\in \R_+^N\}$ be a real-valued random field.
For any $\mathbf{s}=(s_1,\ldots,s_N)\in \R_+^N,\mathbf{t}=(t_1,\ldots,t_N)\in \R_+^N,$ such that $s_k\leq t_k,1\leq k \leq N$ define an increment of $X$ on the rectangle $[\mathbf{s,t}]$ as
\begin{align}
\label{rec:incr}
\Delta_{\mathbf{s}} X(\mathbf{t}) = \sum_{(i_1,\ldots,i_{N})\in \{0,1\}^N}(-1)^{i_1+\cdots + i_N} X(t_1-i_1 (t_1-s_1),\ldots,t_N-i_N (t_N-s_N)).
\end{align}
\end{definition}

In particular, in the case when $N=2,$  rectangular increments have the following form
\[\Delta_{\mathbf{s}} X(\mathbf{t}) =X(t_1,t_2)-X(t_1,s_2)-X(s_1,t_2)+X(s_1,s_2), \mathbf{s,t}\in\R_+^2, s_1\leq t_1,s_2\leq t_2.\]

Rectangular increments may have different probabilistic properties. In our paper, along with the traditional concept of stationary rectangular increments, we also consider two other properties: wide-sense stationarity and mild stationarity.

\begin{definition}
\label{statincr} A random field $\{X(\mathbf{t}),\mathbf{t}\in \R_+^N\}$ has strictly   stationary rectangular increments if for any  $\mathbf{h}\in \R^N_+$
 \(\{\Delta_{\mathbf{h}}X(\mathbf{u}+\mathbf{h}),\mathbf{u}\in \R^N_+\}\stackrel{d}{=}\{\Delta_{\mathbf{0}} X(\mathbf{u}),\mathbf{u}\in \R^N_+\}\).
\end{definition}
\begin{definition}
\label{def:mild}
A random field $X=\{X(\mathbf{t}),\mathbf{t}\in \R_+^N\}$ has  mild stationary rectangular increments if for any fixed $\mathbf{u}\in \R^N_+$ the probability distribution of $\Delta_{\mathbf{h}}X(\mathbf{u}+\mathbf{h})$ does not depend on $\mathbf{h}\in \R^N_+.$
\end{definition}
\begin{definition}
\label{def:wide}
A centered square integrable random field $X=\{X(\mathbf{t}),\mathbf{t}\in \R_+^N\}, N\geq 1$ has  wide-sense stationary rectangular increments if $$\E\left( \Delta_{\mathbf{h}}X(\mathbf{u_1}+\mathbf{h}) \Delta_{\mathbf{h}}X(\mathbf{u_2}+\mathbf{h})\right)=\E \left(\Delta_{\mathbf{0}}X(\mathbf{u_1}) \Delta_{\mathbf{0}}X(\mathbf{u_2})\right)$$ for all $\mathbf{h}, \mathbf{u_1}, \mathbf{u_2}\in \R^N_+.$
\end{definition}

Obviously, Definition \ref{def:mild} is weaker than Definition \ref{statincr}. In case of centered Gaussian random field  Definitions \ref{statincr} and \ref{def:wide} are equivalent due to the fact that finite dimensional distributions of both centered  Gaussian random field and its increments are uniquely determined by the    covariance function. The class of   Gaussian self-similar random fields with strictly   stationary rectangular increments will be denoted by $\mathcal{C}^{\mathbf{H},N}_S,$ and  the class of   Gaussian self-similar random fields with mild   stationary rectangular increments will be denoted by $\mathcal{C}^{\mathbf{H},N}_M$.

We illustrate the importance of random fields from $\mathcal{C}^{\mathbf{H},N}_S$ by the following result with the proof in Appendix.
\begin{theorem}
\label{t1}
Let $\{Y(\mathbf{k}),\mathbf{k}\in \Z^2\}$ be a real-valued strictly stationary random field. Let  $r_{1,n}\to \infty, r_{2,n}\to \infty,$ as $n\to\infty,$ be growing sequences, and $L_1,L_2:\R_+\to \R_+$ be slowly varying functions at $\infty.$ Assume that for $\mathbf{H}=(H_1,H_2)\in (0,1)^2$  there exists a non-trivial random field $V=\{V_{H_1,H_2}(\mathbf{t}),\mathbf{t}\in \R_+^2\}$ such that 
\begin{equation*}  
\left\{\frac{L_1(r_{1,n})L_2(r_{2,n})}{r_{1,n}^{H_1}r_{2,n}^{H_2}}\sum_{\substack{k_1\in [0,t_1 r_{1,n}]\cap \Z \\ k_2\in [0,t_2 r_{2,n}]\cap \Z}}Y(k_1,k_2), (t_1,t_2)\in \R_+^2 \right\}\underset{n\to \infty}{\overset{d}{\to}} \{V_{H_1,H_2}(\mathbf{t}),\mathbf{t}\in \R_+^2\},
\end{equation*}
where $\underset{n\to \infty}{\overset{d}{\to}}$ denotes the limit of finite dimensional distributions. Then $\{V_{H_1,H_2}(\mathbf{t}),\mathbf{t}\in \R_+^2\}$ belongs to $\mathcal{C}^{\mathbf{H},2}_S.$
\end{theorem}

The main purpose of the paper is to characterize   classes $\mathcal{C}^{\mathbf{H},N}_S$ and $\mathcal{C}^{\mathbf{H},N}_M$. As we mentioned above,
$\mathcal{C}^{\mathbf{H},N}_S\subseteq \mathcal{C}^{\mathbf{H},N}_M.$

It is well-known  that  for $N=1$     $\mathcal{C}^{H,1}_S = \mathcal{C}^{H,1}_M$ and any  of these classes  consists of the unique element, namely, of    the fractional Brownian motion with Hurst index $H\in(0,1).$ In the multi-dimensional case, when  $N\geq 2$, we show that the situation is different, namely, we establish  that $\mathcal{C}^{\mathbf{H},N}_S \subset \mathcal{C}^{\mathbf{H},N}_M,$ and the inclusion is strict.

Our main tool is a spectral representation for random fields from $\mathcal{C}^{\mathbf{H},N}_S$ that is established in the paper. Moreover, we find representations of moving average type. With the help of  these representations we   construct various examples of fields from $\mathcal{C}^{\mathbf{H},N}_S, N\geq 2$ which are not fractional Brownian sheets. 

In our previous paper  \cite{mak_mish}, we have provided such example for the class $\mathcal{C}^{\mathbf{H},N}_M,N\geq 2,$ namely, we have proved that $\mathcal{C}^{\mathbf{H},N}_M,N\geq 2$ contains not only the fractional Brownian sheets. In the present paper, we describe the whole class $\mathcal{C}^{\mathbf{H},N}_M$ using a Lamperti transformation and a spectral representation of the stationary Gaussian random fields.

As a by-product, we obtain the new spectral representation of the fractional Brownian motion $\{B^H(t),t\geq 0\}$ with Hurst index $H\in (0,1)$
\begin{align}
\label{fbm:repr0}
B^H(t)=\int_{\R}\frac{t^H e^{i x \ln t}}{|\Gamma(H+ix)|\cdot |\sin (\pi(H+ix))|}\left(\frac{H \Gamma(2H) \sin(\pi H) \cosh(\pi x)}{H^2+x^2}\right)^{1/2}M(dx),
\end{align}
where $M$ is a centered Gaussian random measure on $\R$ with control Lebesgue measure.

We call special attention to the case  $\mathbf{H}=(0.5,\ldots,0.5),$ when a fractional Brownian sheet is a Brownian sheet.
\begin{definition}
A Brownian sheet is a centered real-valued Gaussian random field
$W=\{W(\mathbf{t}),\allowbreak\mathbf{t}\in \R_+^N\}$
with covariance function
\(
\E\bigl(W(\mathbf{t})W(\mathbf{s})\bigr) =
\prod_{i=1}^{N}
\bigl(t_i \ \wedge s_i \bigr),\quad \mathbf{t},\mathbf{s} \in \R_+^N.
\)
\end{definition}
\begin{definition}
A random field $\{X(\mathbf{t}),\mathbf{t}\in \R_+^N\}$ has independent rectangular increments if for all $n\geq 2$ and  $\mathbf{s}_1,\mathbf{t}_1,\ldots,\mathbf{s}_n,\mathbf{t}_n\in \R_+^N,$ such that the rectangles
 $(\mathbf{s}_1,\mathbf{t}_1],\ldots,$ $(\mathbf{s}_n,\mathbf{t}_n]$ have no common internal points, the increments
$\Delta_{\mathbf{s}_1} X(\mathbf{t}_1),\ldots, \Delta_{\mathbf{s}_n} X(\mathbf{t}_n)$
are independent.
\end{definition}
The rectangular increments of the Brownian sheet are both strictly stationary and independent. In this paper, we construct an example $\{Y_{1/2}(\mathbf{t}),\mathbf{t}\in \R^2_+\}$ of Gaussian self-similar random fields with the index $\mathbf{H}=(0.5, 0.5)$ such that their rectangular increments are mild stationary but not independent. Moreover, we prove that $Y_{1/2}$ does not possesses wide-sense stationary rectangular increments, i.e., $Y_{1/2}\in \mathcal{C}^{\mathbf{H},2}_M \setminus \mathcal{C}^{\mathbf{H},2}_S.$

The paper is organized as follows. In Section 2, we  consider Gaussian self-similar random fields with strictly stationary rectangular increments and find their spectral and moving average representations. In Section 3, we consider the class of Gaussian self-similar random fields with mild stationary rectangular increments and we find necessary and sufficient conditions for a Gaussian random field to belong to this class in terms of covariance function of its Lamperti transformation. The results of Section 3 give the method to find the new spectral representation for the fractional Brownian motion, which is presented in Section 4. In this section, we provide also the spectral representation for Gaussian self-similar random fields with mild stationary rectangular increments. In Section 5, we consider the case $H_k=1/2,1\leq k \leq N$ and   provide an example of a self-similar two-parameter Gaussian random field such that its rectangular increments are mild stationary but neither independent nor strictly, consequently, nor wide-sense stationary. In Appendix, we put some auxiliary lemmas.

\section{Gaussian self-similar random fields with strictly stationary increments}
In this section we find the spectral representations of Gaussian random fields from  $\mathcal{C}^{\mathbf{H},N}_S$ and consider some particular examples.

The following statement  is valid not only for Gaussian case.

\begin{proposition}
\label{prop1}
Let a real-valued self-similar random field $\{X(\mathbf{t}),  \mathbf{t}\in \R_+^N\} $ with index $\mathbf{H}\in (0,1)^N$ have mild stationary rectangular increments and finite second moments. Then
\begin{itemize}
\item[(i)] $\E X(\mathbf{t})=0, $ for all $\mathbf{t}\in \R_+^N.$
\item[(ii)] $X(\mathbf{t})=0,$ a.s.   for all $ \mathbf{t}\in \R_+^N: t_1\cdots t_N =0.$
\item[(iii)] $\E[X(\mathbf{t})]^2=\prod_{k=1}^N t_k^{2H_k} \E X^2(1,\ldots,1)$ for all $\mathbf{t}\in \R_+^N.$
\item[(iv)] $\E[\Delta_{\mathbf{s}}X(\mathbf{t})]^2=\prod_{k=1}^N (t_k-s_k)^{2H_k} \E X^2(1,\ldots,1)$   for all $ \mathbf{s},\mathbf{t}\in \R_+^N, s_k\leq t_k, 1\leq k \leq N.$
\end{itemize}
\end{proposition}
\begin{proof} Due to Definition \ref{def:mild} we have
\begin{align*}
\E X(\mathbf{1})&=\E (\Delta_{(0,0,\ldots,0)} X(1,1,\ldots,1)) = \E (\Delta_{(1,0,\ldots,0)} X(2,1,\ldots,1))\\
&=\E X(2,1,\ldots,1)-\E X(1,1,\ldots,1)=(2^{H_1}-1)\E X(\mathbf{1}).
\end{align*}
Therefore, $\E X(\mathbf{1})=0.$
From self-similarity it follows $\E X(\mathbf{t})=t_1^{H_1}\cdots t_N^{H_N}\E X(\mathbf{1}),$ for $\mathbf{t}=(t_1,\ldots,t_N)\in (0,+\infty)^N,$
which gives item $(i)$ of the proposition. Items $(ii)$, $(iii)$ follow from self-similar property.
Item $(iv)$ follows also from the fact that a distribution of a rectangular increment is invariant with respect to translations
\end{proof}
 Obviously,  Proposition \ref{prop1}  is valid for random fields with strictly stationary rectangular increments as well.

Now we focus on spectral representations for random fields with strictly stationary rectangular increments. Recall that for Gaussian case strict and wide-sense properties coincide. Therefore we can apply the results that are valid for the fields with  wide-sense stationary rectangular increments, in particular, we apply the following  theorem that was  proved in \cite[Theorem 2.7]{basse}.
\begin{theorem}
  A real-valued random field $\{X(t),t\in \R^N\}$ has wide-sense stationary rectangular increments if and only if there exists a symmetric measure $F$ on $\R^N$ satisfying $\int_{\R^N} \prod_{j=1}^N \frac{1}{1+z_j^2}F(d\mathbf{z})<\infty$ and a complex-valued random measure $\mathbf{Z}$ with control measure $F$ such that \begin{equation}
\label{spectr:e1}
\Delta_{\mathbf{u}}X(\mathbf{u}+\mathbf{h})=\int_{\R^N} \prod_{j=1}^N\frac{e^{i u_j z_j}\left(e^{i h_j z_j}-1\right)}{i z_j} \mathbf{Z}(d\mathbf{z}), \quad \mathbf{u},\mathbf{h}\in \R_+^N.
\end{equation}
If this is the case then for $\mathbf{s},\mathbf{t},\mathbf{u},\mathbf{v}\in \R_+^N$
\begin{equation}
\label{spectr:e2}
\E(\Delta_{\mathbf{s}}X(\mathbf{s}+\mathbf{u})\Delta_{\mathbf{t}}X(\mathbf{t}+\mathbf{v}))=\int_{\R^N} \prod_{j=1}^N\frac{e^{i (s_j-t_j) z_j}\left(e^{i u_j z_j}-1\right)\left(e^{-i v_j z_j}-1\right)}{z^2_j} F(d\mathbf{z}).
\end{equation}
Moreover, the measures $F$ and $\mathbf{Z}$ are uniquely determined by $X.$ If $X$ is Gaussian, then $\mathbf{Z}$ is a Gaussian random measure.
\end{theorem}

Now we use these representation results in order to characterize  Gaussian random fields from the class $\mathcal{C}^{\mathbf{H},N}_S.$ This is made in the  following theorem which is the main result of this section.
\begin{theorem}
\label{thm:s3}
Let $\{X(t),t\in \R_+^N\}$ be a Gaussian self-similar random field with index $\mathbf{H}=(H_1,\ldots,H_N)\in (0,1)^N,$ $\E X^2(\mathbf{1})=1$ and with strictly stationary rectangular increments. Then $X$ is centered and has covariance function of the form
\begin{align}
\label{thm:eq}
\E X (\mathbf{s})X (\mathbf{t})=\int_{\R^N} \left(\sum_{\mathbf{e}\in \{-1,+1\}^N}K_\mathbf{e} \mathbb{I}\{\mathbf{z}\in Q_{\mathbf{e}}\} \right)\prod_{j=1}^N\frac{\left(e^{i t_j z_j}-1\right)\left(e^{-i s_j z_j}-1\right)}{|z_j|^{2H_j+1}} d\mathbf{z},
\end{align}
where
\begin{equation}
\label{Q:def}
Q_\mathbf{e}=\{ (e_1 y_1,\ldots,e_N y_N) : y_j>0,1\leq j \leq N \}, \mathbf{e}=(e_1,\ldots,e_N)\in \{-1,+1\}^N,
\end{equation}
and
$K_\mathbf{e},\mathbf{e}\in \{-1,+1\}^N$ are  non-negative constants satisfying the following relations
\begin{align}
\label{K:def1}
K_\mathbf{e}&=K_\mathbf{-e},\mathbf{e}\in \{-1,+1\}^N,\\
\label{K:def2}\sum_{\mathbf{e}\in \{-1,+1\}^N}K_\mathbf{e} &=\prod_{j=1}^N\frac{\Gamma(1+2H_j)\sin (\pi H_j)}{\pi}.
\end{align}
\end{theorem}
\begin{proof}
It  follows from Proposition \ref{prop1} $(i)$ and  $(ii)$ that $X$ is centered and $X(\mathbf{t})=0$ a.s. for all $\mathbf{t}\in \R_+^N$ such that $t_1\cdots t_N=0.$ Then $X(\mathbf{t})$ equals  $\Delta_\mathbf{0}X(\mathbf{t})$ a.s. Moreover, we write now the spectral representation  of $X$ and its covariance function. From \eqref{spectr:e1} we have for any $\mathbf{t}\in \R_+^N$
\begin{equation}
\label{thm:s3:e1}
    X(\mathbf{t})\stackrel{a.s.}{=}\Delta_\mathbf{0}X(\mathbf{t}) \stackrel{eq. \eqref{spectr:e1}}{=}\int_{\R^N} \prod_{j=1}^N\frac{\left(e^{i t_j z_j}-1\right)}{i z_j} \mathbf{Z}(d\mathbf{z}),
\end{equation}
where $\mathbf{Z}$ is a centered Gaussian random measure  uniquely determined by its control measure $F.$ Now we describe the structure of the measure $F.$

From \eqref{spectr:e2} we have for any $\mathbf{t},\mathbf{s}\in \R_+^N$
\begin{equation}
\label{thm:s3:e2}
    \E X(\mathbf{t}) X(\mathbf{s})=\E(\Delta_\mathbf{0}X(\mathbf{t})\Delta_\mathbf{0}X(\mathbf{s})) \stackrel{eq. \eqref{spectr:e2}}{=}\int_{\R^N} \prod_{j=1}^N\frac{\left(e^{i t_j z_j}-1\right)\left(e^{-i s_j z_j}-1\right)}{z^2_j} F(d\mathbf{z}).
\end{equation}
From self-similarity we get  for any $\mathbf{t,s}\in \R_+^N$ and for all $(a_1,\ldots,a_N)\in (0,+\infty)^N$ the identity
\begin{equation}\label{thm:s3:e3}
\E X(a_1t_1,\ldots,a_N t_N) X(a_1s_1,\ldots,a_N s_N)=a_1^{2H_1}\cdots a_N^{2H_N}\E X(\mathbf{t}) X(\mathbf{s}).
\end{equation}
We rewrite the left-hand side with the help of spectral representation \eqref{thm:s3:e2}:
\begin{align*}
    \E X(a_1t_1,\ldots,a_N t_N) & X(a_1s_1,\ldots,a_N s_N)\\
    &=\int_{\R^N} \prod_{j=1}^N\frac{\left(e^{i a_j t_j z_j}-1\right)\left(e^{-i a_j s_j z_j}-1\right)}{z^2_j} F(d\mathbf{z})\\
    &|\text{change variables }x_j=a_j z_j|\\&=\int_{\R^N} \prod_{j=1}^N\frac{\left(e^{it_j x_j}-1\right)\left(e^{-i s_j x_j}-1\right)}{x^2_j} F_a(d\mathbf{x}),
\end{align*}
where measure $F_a$ is given for any $B\in \mathcal{B}(\R^N)$ by
\begin{equation}
    \label{Fa:def}
    F_a(B)=a_1^2\cdots a_N^{2} F(B_a), \text{ where } B_a=\left\{\left(\frac{z_1}{a_1},\ldots,\frac{z_N}{a_N}\right),(z_1,\ldots,z_N)\in B\right\}.
\end{equation}
    Therefore, relation \eqref{thm:s3:e3} has the following form
$$\int_{\R^N} \prod_{j=1}^N\frac{\left(e^{it_j x_j}-1\right)\left(e^{-i s_j x_j}-1\right)}{x^2_j} F_a(d\mathbf{x})=\int_{\R^N} \prod_{j=1}^N\frac{\left(e^{it_j x_j}-1\right)\left(e^{-i s_j x_j}-1\right)}{x^2_j} a_j^{2H_j} F(d\mathbf{x}).$$
The uniqueness of spectral representation gives
\begin{equation}
    \label{Fa:eq}
F_a(B)=a_1^{2H_1}\cdots a_N^{2H_N} F(B), \text{ for any } B\in \mathcal{B}(\R^N) \text{ such that }F(B)<\infty
\end{equation}
and for all $a_j>0, 1\leq j \leq N.$

Since $\int_{\R^N} \prod_{j=1}^N \frac{1}{1+z_j^2}F(d\mathbf{z})<\infty,$ we see that $F((-\mathbf{y},\mathbf{y}])<\infty$ for any $\mathbf{y}\in (0,+\infty)^N.$ Let us take $B=(\mathbf{0},\mathbf{y}],\mathbf{y}=(y_1,\ldots,y_N)$ and $a_j=y_j,$ then $B_a=(0,1]^N$ and $F_a(B)=y_1^{2}\cdots y_N^{2}F\left((0,1]^N\right).$  Hence,  it follows from \eqref{Fa:eq} that
$$F\left((0,1]^N\right) \prod_{j=1}^N y_j^{2}=F\left((\mathbf{0},\mathbf{y}]\right)\prod_{j=1}^N y_j^{2H_j}, \mathbf{y}\in (0,+\infty)^N $$
or $$F((\mathbf{0},\mathbf{y}])=F\left((0,1]^N\right)\prod_{j=1}^N y_j^{2(1-H_j)}, \mathbf{y}\in (0,+\infty)^N.$$
Therefore, at any point $\mathbf{y}\in (0,+\infty)^N$ there exists a density with respect to Lebesgue measure and
\begin{equation}
\label{F:def}
F(d\mathbf{y})=K_\mathbf{1}\prod_{j=1}^N y_j^{1-2H_j}d\mathbf{y},
\end{equation}
where $K_\mathbf{1}$ is a non-negative  constant.

Applying similar arguments we can show that measure $F$ has a density on any set $Q_\mathbf{e}=\{ (e_1 y_1,\ldots,e_N y_N) : y_j>0,1\leq j \leq N \},$  $\mathbf{e}=(e_1,\ldots,e_N)\in \{-1,1\}^N$ and
$F(d\mathbf{y})=K_\mathbf{e}\prod_{j=1}^N |y_j|^{1-2H_j}d\mathbf{y}, \mathbf{y}\in Q_\mathbf{e}$ with non-negative constants $K_\mathbf{e}.$ Hence, the measure $F$ has the following form
\begin{equation}
\label{F:def:1}
F(d\mathbf{y})=\sum_{\mathbf{e}\in \{-1,+1\}^N}\frac{K_\mathbf{e} \mathbb{I}\{\mathbf{y}\in Q_{\mathbf{e}}\}}{|y_1|^{2H_1-1}\cdots |y_N|^{2H_N-1}}d\mathbf{y}, \mathbf{y}\in \R^N.
\end{equation}
Thus, we obtain statement \eqref{thm:eq}.

Since $F$ is symmetric, $K_\mathbf{e}=K_{-\mathbf{e}}, \mathbf{e}\in \{-1,+1\}^N,$ i.e., we get \eqref{K:def1}.  Moreover, from identity $\E X^2 (\mathbf{1})=1$ and representation \eqref{thm:s3:e2} we obtain the following relation.
\begin{align*}
1=\E X^2 (\mathbf{1})&= \int_{\R^N} \left(\sum_{\mathbf{e}\in \{-1,+1\}^N}K_\mathbf{e} \mathbb{I}\{\mathbf{y}\in Q_{\mathbf{e}}\} \right)\prod_{j=1}^N\frac{\left|e^{i z_j}-1\right|^2}{|z_j|^{2H_1+1}} {d\mathbf{z}}\\
&=\left(\sum_{\mathbf{e}\in \{-1,+1\}^N}K_\mathbf{e} \right) \int_{\R_+^N} \prod_{j=1}^N\frac{\left|e^{i z_j}-1\right|^2}{|z_j|^{2H_j+1}} {d\mathbf{z}}\\
    &\overset{eq. \eqref{rem1:eq}}{=}\left(\sum_{\mathbf{e}\in \{-1,+1\}^N}K_\mathbf{e} \right)\prod_{j=1}^N\frac{\pi}{\Gamma(1+2H_j)\sin (\pi H_j)}.
\end{align*}
This gives relation \eqref{K:def2}.
\end{proof}
The condition of symmetry \eqref{K:def1} guaranties that  covariance function  \eqref{thm:eq} is a real-valued function.

If the measure $F$ is symmetric with respect to coordinate axes, i.e., $K_\mathbf{e}=K_\mathbf{1},\mathbf{e}\in \{-1,+1\}^N,$ then representation \eqref{thm:eq} coincides with the representation of the covariance function of a fractional Brownian sheet. Consequently, the spectral representation of $X$ coincides in this case with the spectral representation of a fractional Brownian motion, which is given in the following proposition.
\begin{proposition}[\cite{alp}]
A fractional Brownian sheet $\{B^\mathbf{H}(\mathbf{t}),\mathbf{t}\in \R^N_+\}$ with Hurst index $\mathbf{H}=(H_1,\ldots,H_N)$ has the harmonizable representation
\begin{equation}
\label{harm:repr:2}
B^\mathbf{H}(\mathbf{t})=\int_{\R^N} \prod_{j=1}^N\frac{e^{it_j x_j}-1}{ix_j }\frac{c_1(H_j)}{|x_j|^{H_j-1/2}} \widetilde{M}(d\mathbf{x}), \mathbf{t} \in \R_+^N,
\end{equation}
where
\begin{equation}
\label{CH:def}
    c_1(H)=\left(H\Gamma(2H) \frac{\sin (\pi H)}{\pi} \right)^{1/2},
\end{equation}
  and $\widetilde{M}$ is the  Fourier transform of some Wiener measure $M.$
\end{proposition}
 A fractional Brownian sheet has strictly stationary rectangular increments. It is proved in \cite{alp} but this fact can be also derived from representation \eqref{harm:repr:2}.

Let us find now explicit forms of covariance functions for Gaussian random fields from $\mathcal{C}^{\mathbf{H},N}_S.$
\begin{corollary}
\label{cor1}Let assumptions of Theorem \ref{thm:s3} be fulfilled. Denote by \linebreak
$$\gamma_\mathbf{e}=K_\mathbf{e}\prod_{j=1}^N\frac{\pi}{\Gamma(1+2H_j)\sin (\pi H_j)}=K_\mathbf{e}\prod_{j=1}^N\frac{1}{2 c_1^2(H_j)},\mathbf{e}\in \{-1,1\}^N.$$ With these notations,  conditions \eqref{K:def1} and \eqref{K:def2} are equivalent to
$\gamma_\mathbf{e}=\gamma_{-\mathbf{e}}, \mathbf{e}\in \{-1,1\}^N$ and $\sum_{\mathbf{e}\in \{-1,1\}^N}\gamma_\mathbf{e}=1$, respectively.
Assume additionally that $H_j\neq 1/2,1\leq j \leq N,$ then  $X$ has covariance function
\begin{align}
\label{cor1:eq}\E X (\mathbf{s})X (\mathbf{t})&=\sum_{\mathbf{e}\in \{-1,+1\}^N}\frac{\gamma_\mathbf{e}}{2^N} \prod_{j=1}^N\left[\left( t_j^{2H_j} + s_j^{2H_j}-|t_j-s_j|^{2H_j}\right)\right.\\
\nonumber&\left.+i e_j  \tan(\pi H_j)\left(-t_j^{2H_j} +  s_j^{2H_j}+ \sign (t_j-s_j) |t_j-s_j|^{2H_j}\right)\right].
\end{align}
\end{corollary}
\begin{proof}
Let us write down the covariance function of $X$ from \eqref{thm:eq}. For any  $\mathbf{t},\mathbf{s}\in \R_+^N$ we have that
\begin{align*}
\E X (\mathbf{s})X (\mathbf{t})&= \int_{\R^N} \left(\sum_{\mathbf{e}\in \{-1,+1\}^N}K_\mathbf{e} \mathbb{I}\{\mathbf{z}\in Q_{\mathbf{e}}\} \right)\prod_{j=1}^N\frac{\left(e^{i t_j z_j}-1\right)\left(e^{-i s_j z_j}-1\right)}{|z_j|^{2H_j+1}} d\mathbf{z}\\
&= \sum_{\mathbf{e}\in \{-1,+1\}^N}K_\mathbf{e} \int_{\R^N\cap Q_\mathbf{e}} \prod_{j=1}^N\frac{\left(e^{i t_j z_j}-1\right)\left(e^{-i s_j z_j}-1\right)}{|z_j|^{2H_j+1}} d\mathbf{z}\\
&=|\text{change variables }z_j=e_j y_j| \\
&=\sum_{\mathbf{e}\in \{-1,+1\}^N}K_\mathbf{e} \int_{\R_+^N} \prod_{j=1}^N\frac{\left(e^{i t_j  e_j y_j}-1\right)\left(e^{-i s_j e_j y_j}-1\right)}{|y_j|^{2H_j+1}} d\mathbf{y}\\
&\overset{\eqref{lemma1:eq}}{=}\sum_{\mathbf{e}\in \{-1,+1\}^N}K_\mathbf{e}\prod_{j=1}^N\frac{\pi}{\Gamma(1+2H_j)\sin (2 \pi H_j)}\\
&\times \prod_{j=1}^N\left(e^{-i \pi H_j  e_j } t_j^{2H_j} + e^{i \pi H_j  e_j}s_j^{2H_j}-e^{-i \pi H_j  e_j \sign (t_j-s_j)}|t_j-s_j|^{2H_j}\right)\\
&=\sum_{\mathbf{e}\in \{-1,+1\}^N}\gamma_\mathbf{e}\prod_{j=1}^N\frac{1}{2\cos ( \pi H_j)}\\
&\times \prod_{j=1}^N\left[\cos(\pi H_j)\left( t_j^{2H_j} + s_j^{2H_j}-|t_j-s_j|^{2H_j}\right)\right.\\
&\left.+i e_j  \sin(\pi H_j)\left(-t_j^{2H_j} +  s_j^{2H_j}+ \sign (t_j-s_j) |t_j-s_j|^{2H_j}\right)\right]\\
&=\sum_{\mathbf{e}\in \{-1,+1\}^N}\frac{\gamma_\mathbf{e}}{2^N} \prod_{j=1}^N\left[\left( t_j^{2H_j} + s_j^{2H_j}-|t_j-s_j|^{2H_j}\right)\right.\\
&\left.+i e_j  \tan(\pi H_j)\left(-t_j^{2H_j} +  s_j^{2H_j}+ \sign (t_j-s_j) |t_j-s_j|^{2H_j}\right)\right].
\end{align*}
\end{proof}

\begin{corollary} Let assumptions of Theorem \ref{thm:s3} be satisfied and $H_1=\ldots=H_N=1/2,$ then  $X$ has covariance function
\begin{align}
\label{cor2:eq}
&\E X (\mathbf{s})X (\mathbf{t})\\
\nonumber&=\sum_{\mathbf{e}\in \{-1,+1\}^N}K_\mathbf{e} \prod_{j=1}^N\left[{\pi}\left( t_j \wedge s_j\right)+i e_j \left(t_j \log t_j - s_j \log s_j - ( t_j-s_j)\log |t_j-s_j|\right)\right]
\end{align}
and \eqref{K:def2} has the following form
$\sum_{\mathbf{e}\in \{-1,+1\}^N}K_\mathbf{e}=1/\pi^N.$
\end{corollary}
\begin{proof}
Let us write down the covariance function of $X.$ For any $\mathbf{t},\mathbf{s}\in \R_+^N$ with the  change of variables $z_j  =e_jy_j$ in \eqref{thm:eq}, we have that
\begin{align*}
\E X (\mathbf{s})X (\mathbf{t})&=\sum_{\mathbf{e}\in \{-1,+1\}^N}K_\mathbf{e} \int_{\R_+^N} \prod_{j=1}^N\frac{\left(e^{i t_j  e_j y_j}-1\right)\left(e^{-i s_j e_j y_j}-1\right)}{y_j^{2}} d\mathbf{y}\\
&\overset{\text{Lemma } \ref{lemma2}}{=}
\sum_{\mathbf{e}\in \{-1,+1\}^N}K_\mathbf{e} \prod_{j=1}^N\left[\frac{\pi}{2}\left( |e_jt_j| + |e_j s_j|-|e_j t_j-e_j s_j|\right)\right.\\
&\left.+i  \left(e_j t_j \log |e_j t_j| - e_j s_j \log |e_j s_j| - (e_j t_j-e_j s_j)\log |e_j t_j-e_j s_j|\right)\right]\\
&= \sum_{\mathbf{e}\in \{-1,+1\}^N}K_\mathbf{e} \prod_{j=1}^N\left[\frac{\pi}{2}\left( t_j + s_j-| t_j- s_j|\right)\right.\\
&\left.+i e_j \left(t_j \log t_j - s_j \log s_j - ( t_j-s_j)\log |t_j-s_j|\right)\right].
\end{align*}
\end{proof}

In general case, when only some of $H_k$ are equal to $1/2$ and others not, the covariance function of $X$ has the form
$$\E X (\mathbf{s})X (\mathbf{t})=\sum_{\mathbf{e}\in \{-1,+1\}^N}K_\mathbf{e}\prod_{j=1}^N P_{H_j}(t_j,s_j,e_j),$$
where
\begin{align*}
    &P_{H_j}(t_j,s_j,e_j)=\frac{\pi}{2\Gamma(1+2H_j)\sin (\pi H_j)}\left[\left( t_j^{2H_j} + s_j^{2H_j}-|t_j-s_j|^{2H_j}\right)\right.\\
\nonumber&\left.+i e_j  \tan(\pi H_j)\left(-t_j^{2H_j} +  s_j^{2H_j}+ \sign (t_j-s_j) |t_j-s_j|^{2H_j}\right)\right], \text{ if } H_j\neq \frac12,
\end{align*}
and
\begin{align*}
    &P_{H_j}(t_j,s_j,e_j)=\pi(t_j \wedge s_j)+i e_j \left(t_j \log t_j - s_j \log s_j - ( t_j-s_j)\log |t_j-s_j|\right),\text{ if } H_j=\frac12.
\end{align*}

Let us consider the simpler case $N=2.$
\begin{remark}
Let $H_1\neq \frac12,H_2\neq \frac12$ and  $\gamma\in[-1,1].$ Then \eqref{cor1:eq} turns into
\begin{align*}
\E X (\mathbf{s})X (\mathbf{t})&=\frac{1}{4} \prod_{j=1}^2\left( t_j^{2H_j} + s_j^{2H_j}-|t_j-s_j|^{2H_j}\right)\\
& +\frac{\gamma}{4}\prod_{j=1}^2 \tan(\pi H_j)\left(-t_j^{2H_j} +  s_j^{2H_j}+ \sign (t_j-s_j) |t_j-s_j|^{2H_j}\right).
\end{align*}
\end{remark}
\begin{remark}
Let $H_1=H_2=\frac12$ and  $\gamma\in[-1,1].$ Then \eqref{cor2:eq} turns into
\begin{equation}
\label{rem2:eq}\E X (\mathbf{s})X (\mathbf{t})=\left( t_1 \wedge s_1\right) \left( t_2 \wedge s_2\right)+\frac{\gamma}{\pi^2} \prod_{j=1}^2\left(t_j \log t_j - s_j \log s_j - ( t_j-s_j)\log |t_j-s_j|\right).
\end{equation}
\end{remark}

Moreover, covariance function \eqref{rem2:eq} are the same as for the two-dimensional fractional Brownian motion described in \cite{surgalis09}.

\begin{remark}
Let $H_1=\frac12, H_2\neq\frac12$ and  $\gamma\in[-1,1].$ Then the covariance function of $X$ equals
\begin{align*}
\E X (\mathbf{s})X (\mathbf{t})&=\frac{1}{2} \left( t_1 \wedge s_1\right)\left( t_2^{2H_2} + s_2^{2H_2}-|t_2-s_2|^{2H_2}\right)\\
& +\frac{\gamma \tan(\pi H_2)}{2\pi}\left(t_1 \log t_1 - s_1 \log s_1 - ( t_1-s_1)\log |t_1-s_1|\right)\\
&\times\left(-t_2^{2H_2} +  s_2^{2H_2}+ \sign (t_2-s_2) |t_2-s_2|^{2H_2}\right).
\end{align*}
\end{remark}

Hence, we can write the general form of harmonizable representations for Gaussian random fields from $\mathcal{C}^{\mathbf{H},N}_S.$  The following result follows directly from Theorem \ref{thm:s3}.
\begin{theorem}
\label{thm:s6}
Let $\{X(t),t\in \R_+^N\}$ be a Gaussian self-similar random field with index $\mathbf{H}=(H_1,\ldots,H_N)\in (0,1)^N,$ $\E X^2(\mathbf{1})=1$ and with strictly stationary rectangular increments. Then $X$ has the following representation
\begin{align}
\label{thm6:eq}
X (\mathbf{t})=\int_{\R^N} \left(\sum_{\mathbf{e}\in \{-1,+1\}^N}\sqrt{K_\mathbf{e}}\exp(i\varphi_{\mathbf{e}}) \mathbb{I}\{\mathbf{z}\in Q_{\mathbf{e}}\} \prod_{j=1}^N\frac{e^{i t_j z_j}-1}{i e_j} |z_j|^{-H_j-1/2}\right) \widetilde{M}(d\mathbf{z}),
\end{align}
where $\widetilde{M}$ is the Fourier transform of Brownian measure $M,$ and $Q_\mathbf{e}$ are defined in \eqref{Q:def}, $K_\mathbf{e}$ are non-negative constants satisfying relations \eqref{K:def1} and \eqref{K:def2}, $\varphi_\mathbf{e}\in \R$ such that $\varphi_\mathbf{-e}=-\varphi_\mathbf{e}, \mathbf{e}\in \{-1,+1\}^N.$
\end{theorem}
\begin{proof}
From the proof of Theorem \ref{thm:s3} we get that $X$ has integral representation \eqref{thm:s3:e1} with respect to random measure $\mathbf{Z}.$ Since we consider Gaussian random fields, the measure $\mathbf{Z}$ is Gaussian and can be rewritten as $\mathbf{Z}(d\mathbf{y})=f(\mathbf{y})\widetilde{M}(d \mathbf{y}),$ where $f:\R^d\to \mathbb{C}.$ 
Then $Z$ has  control measure $F$ satisfying $F(d\mathbf{y})=|f(\mathbf{y})|^2d\mathbf{y},\mathbf{y}\in \R^N.$ From $\eqref{F:def:1}$ we get
\begin{equation*}
|f(\mathbf{y})|^2=\sum_{\mathbf{e}\in \{-1,+1\}^N}\frac{K_\mathbf{e} \mathbb{I}\{\mathbf{y}\in Q_{\mathbf{e}}\}}{|y_1|^{2H_1-1}\cdots |y_N|^{2H_N-1}}, \mathbf{y}\in \R^N.
\end{equation*}
and consequently
$$f(\mathbf{y})=\sum_{\mathbf{e}\in \{-1,+1\}^N}\frac{\sqrt{K_\mathbf{e}}\exp(i\varphi_{\mathbf{e}}) \mathbb{I}\{\mathbf{y}\in Q_{\mathbf{e}}\}}{|y_1|^{H_1-1/2}\cdots |y_N|^{H_N-1/2}}, \mathbf{y}\in \R^N.$$
Symmetry condition $Z(-A)=\overline{Z(A)},A\in \mathcal{B}_b(\R^N)$ gives that $f(-y)=\overline{f(y)}.$ This relation and symmetry \eqref{K:def1} of $K_\mathbf{e}$ give that  $\varphi_{-\mathbf{e}}=-\varphi_{\mathbf{e}},\mathbf{e}\in \{-1,+1\}^N.$
Thus, $X$ has representation \eqref{thm6:eq}.
\end{proof}

Let us now consider the representations of moving average type. Further denote $x_+:=\max\{x,0\}$ and $x_-:=-\min\{x,0\}$ for $x\in \R.$ For the fractional Brownian sheet we have an analogue of Mandelbrot-van-Ness representation.
\begin{proposition}[\cite{alp}]
A fractional Brownian sheet $\{B^\mathbf{H}(\mathbf{t}),\mathbf{t}\in \R^N_+\}$ with Hurst index $\mathbf{H}=(H_1,\ldots,H_N)$ has the moving average representation
\begin{equation}
\label{moving:repr:2}
B^\mathbf{H}(\mathbf{t})=c_2(H_1)\cdots c_2(H_N) \int_{\R^N} \prod_{j=1}^N \left((t_j-x_j)_+^{H_j-1/2}-(-x_j)_+^{H_j-1/2}\right) M(d\mathbf{x}), \mathbf{t} \in \R_+^N,
\end{equation}
where 
\begin{equation}
\label{CH2:def}
    c_2(H)= \frac{\left(\Gamma(1+2H)\sin (\pi H) \right)^{1/2}}{\Gamma(H+1/2)},
\end{equation}
  and $M$ is a  Wiener measure.
\end{proposition}
For arbitrary Gaussian random fields from $\mathcal{C}^{\mathbf{H},N}_S$ we have the following result.
\begin{theorem}
\label{thm:s7}
Let a random field $\{X(\mathbf{t}),\mathbf{t}\in \R^N_+\}$ satisfies assumptions of Theorem \ref{thm:s6} and $H_k\in (0,1),H_k\neq 1/2,$ $1\leq k \leq N.$ Then
\begin{equation}
\label{g}    X(\mathbf{t})=\int_{\R^N} g(\mathbf{t},\mathbf{x}) M(\mathbf{x}),~\mathbf{t}\in \R_+^N,    
\end{equation}
where
\begin{align}
\label{thm:s7:eq}
\nonumber g(\mathbf{t},\mathbf{x})&=\sum_{\mathbf{e}\in \{-1,+1\}^N}\sqrt{K_\mathbf{e}}\exp(i\varphi_{\mathbf{e}}) \prod_{j=1}^N \frac{\Gamma\left(\frac12-H_j\right)}{\sqrt{2\pi}}\\
\nonumber &\times\prod_{j=1}^N\left[\left((t_j-x_j)_+^{H_j-1/2}-(-x_j)_+^{H_j-1/2}\right)\exp\left(-\frac{i\pi e_j}{2}\left(H_j+\frac12\right)\right)\right.\\
   &\left.-\left((t_j-x_j)_-^{H_j-1/2}-(-x_j)_-^{H_j-1/2}\right)\exp\left(\frac{i\pi e_j}{2}\left(H_j+\frac12\right)\right)\right],
\end{align}
$K_\mathbf{e},\mathbf{e}\in \{-1,+1\}$ satisfy relations \eqref{K:def1} and \eqref{K:def2},
and $\varphi_\mathbf{-e}=-\varphi_\mathbf{e}, \mathbf{e}\in \{-1,+1\}^N.$
\end{theorem}
\begin{proof} 
 Let  $\tilde{g}(\mathbf{t},\cdot)$ be the Fourier transform of function $g(\mathbf{t},\cdot).$ Due to \cite[Proposition 7.2.7]{sam} if $\tilde{g}(\mathbf{t},-\mathbf{x})=\overline{\tilde{g}(\mathbf{t},\mathbf{x})},$  $\mathbf{x}\in \R^N,$ then we have 
$\{X(\mathbf{t}),\mathbf{t}\in \R_+^N\}\stackrel{d}{=}\{\int_{\R^N} \tilde{g}(\mathbf{t},\mathbf{x}) \widetilde{M}(d \mathbf{y}),\mathbf{t}\in \R_+^N\}.$
So, we find the function $g$ as the inverse Fourier transform of the integrand in \eqref{thm6:eq}, i.e., $g(\mathbf{t},\mathbf{x})$ equals
\begin{align*}
&\frac{1}{(2\pi)^{N/2}}\int_{\R^N}e^{-i<\mathbf{x},\mathbf{y}>}\left(\sum_{\mathbf{e}\in \{-1,+1\}^N}\sqrt{K_\mathbf{e}}\exp(i\varphi_{\mathbf{e}}) \mathbb{I}\{\mathbf{y}\in Q_{\mathbf{e}}\} \prod_{j=1}^N\frac{e^{i t_j y_j}-1}{i e_j} |y_j|^{-H_j-1/2} \right)d\mathbf{y}\\
&=\sum_{\mathbf{e}\in \{-1,+1\}^N}\sqrt{K_\mathbf{e}}\exp(i\varphi_{\mathbf{e}}) \prod_{j=1}^N\frac{1}{\sqrt{2\pi}}\int_{\R_+}e^{-i e_j x_jy_j}\frac{e^{i e_j t_j y_j}-1}{i e_j y_j^{H_j+1/2}} d y_j.
\end{align*}
Application of Lemma \ref{lemma3} ends the proof.
\end{proof}
In order to write simplified version of representation \eqref{g} we consider the  case $N=2.$
\begin{corollary}
\label{moving}
Let a random field $\{X(\mathbf{t}),\mathbf{t}=(t_1,t_2)\in \R_+^2\}$ be  given by
\begin{align*}
 X(\mathbf{t})&=d_0 c_2(H_1)c_2(H_2)\int_{\R^N}  \prod_{j=1}^N \left((t_j-x_j)_+^{H_j-1/2}-(-x_j)_+^{H_j-1/2}\right) M(d \mathbf{x})\\
   &+d_1 c_2(H_1)c_2(H_2)\int_{\R^N}  \prod_{j=1}^N \left((t_j-x_j)_-^{H_j-1/2}-(-x_j)_-^{H_j-1/2}\right) M(d \mathbf{x}),
\end{align*}
where $H_1\neq \frac12, H_2\neq \frac12,$ and the constants $d_0,d_1$ satisfy
\begin{equation}
\label{dd} d_0^2+2d_0d_1\sin(\pi H_1)\sin(\pi H_2)+d_1^2=1.
\end{equation}
Then $X$ is a centered Gaussian self-similar random field with index $(H_1,H_2),$ $\E X^2(\mathbf{1})=1,$ and $X$ possesses stationary rectangular increments.
\end{corollary}
\begin{proof}
In Appendix.
\end{proof}
Thus, if the function $g$ given \eqref{thm:s7:eq} from representation of random field $X\in\mathcal{C}^{\mathbf{H},2}_S.$  ``depends on the past'', i.e., $d_1=0,$ then $|d_0|=1$ and $X$ is the fractional Brownian sheet with Hurst index $(H_1,H_2).$

We have the similar results for the case $H_j=1/2,1\leq j \leq N.$
\begin{theorem}
\label{thm:s8}
Let a random field $\{X(\mathbf{t}),\mathbf{t}\in \R^N_+\}$ satisfies assumptions of Theorem \ref{thm:s6} and $H_j= 1/2,$ $1\leq j \leq N.$ Then for $\mathbf{t}\in \R^N$
\begin{equation}
\label{g2}    X(\mathbf{t})=\int_{\R^N} \sum_{\mathbf{e}\in \{-1,+1\}^N}\frac{\sqrt{K_\mathbf{e}}\exp(i\varphi_{\mathbf{e}})}{(2\pi)^{N/2}} \prod_{j=1}^N \left[\pi\mathbb{I}_{[0,t_j]}(x_j)+e_j i\left( \log\frac{|t_j-x_j|}{|x_j|}\right)\right] M(\mathbf{x}),
\end{equation}
where $K_\mathbf{e},\mathbf{e}\in \{-1,+1\}$ satisfy relations \eqref{K:def1} and $\sum_{\mathbf{e}\in \{-1,+1\}^N}K_\mathbf{e}=1/\pi^N.$
\end{theorem}
\begin{proof}The proof repeats the proof of Theorem \ref{thm:s7} together with the application of Lemma \ref{lemma4}.
\end{proof}
In the case $N=2$ and $H_1=H_2=1/2$ we have the following result in the spirit of Corollary \ref{moving}.

\begin{corollary}
\label{moving-1}
Let a random field $\{X(\mathbf{t}),\mathbf{t}=(t_1,t_2)\in \R_+^2\}$ be  given by
\begin{align*}
 X(\mathbf{t})&=d_0 M([\mathbf{0},\mathbf{t}])+\frac{d_1}{\pi^2}\int_{\R^2} \log\frac{|t_1-x_1|}{|x_1|} \log\frac{|t_2-x_2|}{|x_2|} M(d \mathbf{x}),
\end{align*}
where the constants $d_0,d_1$ satisfy $d_0^2+d_1^2=1.$
Then $X$ is a Gaussian self-similar random field with index $(0.5,0.5)$ and $X$ possesses stationary rectangular increments.
\end{corollary}



\section{Gaussian self-similar random fields with mild stationary rectangular increments}

In this section, we characterize the class of Gaussian self-similar random fields from $\mathcal{C}^{\mathbf{H},N}_M,$ $N>1,$ with the necessary and sufficient conditions that must be met by  their covariance functions. This is established with the help of Lamperti transformation.

But at first, note that for the case $N=1,$ $\mathcal{C}^{H,1}_S=\mathcal{C}^{H,1}_M.$ In this case the description of this class is very simple and is contained in the following remark.
\begin{remark}
\label{remark1}
A fractional Brownian motion $B^H$ is a self-similar process with index $H$ and has strictly stationary as well as mild stationary increments. Moreover, $B^H$ is an unique Gaussian process from $\mathcal{C}^{H,1}_M$. Indeed, let  $\{X(t),t \in \R_+\}$ be a square integrable   real-valued self-similar process with index $H\in (0,1)$ and with mild stationary increments, then
\begin{align}
\nonumber\E[X(t)X(s)]&=\frac{1}{2}\E\left(X^2(t)+X^2(s)-(X(t)-X(s))^2\right)\\
\nonumber &=\frac{1}{2}\left(\E X^2(t)+\E X^2(s)-\E X^2(t-s)\right)\\
\label{eq:proc}&=\frac{1}{2}\left(t^{2H}+s^{2H}-|t-s|^{2H}\right)\E[X(1)]^2,\quad t,s\in \R_+.
\end{align}
Furthermore, if the second moments of $\{X(t),t\in \R_+\}$ are finite, then $X$ is centered due to Proposition \ref{prop1}. 
Hence, all square integrable self-similar processes with mild-stationary increments have the same covariance function \eqref{eq:proc}.
\end{remark}

We use a one-to-one correspondence between self-similar and strictly stationary random fields. This is carried out by the Lamperti transformation.
\begin{definition}
A random field $\{Z(\mathbf{t}),\mathbf{t}\in \R^N\}$ is called strictly stationary, if for all $\mathbf{u}\in \R^N$
$\{Z(\mathbf{t}+\mathbf{u}),\mathbf{t} \in \R^N\}\stackrel{d}{=}\{Z(\mathbf{t}),\mathbf{t} \in \R^N\}.$
\end{definition}
\begin{definition}
The Lamperti transformation with index $\mathbf{H}=(H_1,\ldots,H_N)$ of a random field $\{X(\mathbf{t}),\mathbf{t}\in \R^N_+\}$  is a random field $\tau_{\mathbf{H}}X,$ defined by
\begin{equation}
\label{lamp}
Z(\mathbf{t}) = \tau_{\mathbf{H}}X(\mathbf{t}):=e^{-H_1t_1}\cdots e^{-H_N t_N} X(e^{t_1},\ldots ,e^{t_N}), \mathbf{t}\in \R^N.
\end{equation}
\end{definition}
It follows from \cite[Proposition 2.1.1]{taqqu}  that if $X$ is self-similar with index $\mathbf{H}\in (0,1)^N$, then $\tau_{\mathbf{H}}X$ is strictly stationary. The inverse statement also holds: for any strictly stationary random field $\{Z(\mathbf{t}),\mathbf{t}\in \R^N\}$ a field $X,$ defined as
\begin{equation}
\label{defLam}
X(\mathbf{s})=\tau_{\mathbf{H}}^{(-1)}Z(\mathbf{t}):=\begin{cases} s_1^{H_1}\cdots s_N^{H_N}Z(\ln s_1,\ldots,\ln s_N),& \mathbf{s}\in (0,+\infty)^N,\\
0, &\mathbf{s}\in \R_+^N \setminus (0,+\infty)^N,\end{cases}
\end{equation}
is self-similar with index  $\mathbf{H}=(H_1,\ldots,H_N).$ Obviously, a random field $\tau_{\mathbf{H}}X$ is centered if and only if $X$ centered. Now, let $X$ be a centered self-similar square integrable random field with index $\mathbf{H}.$ Then the covariance function $C:\R^N\to \R$ of the field $\tau_{\mathbf{H}}X$ is determined by covariance function $K:\R_+^N\times \R_+^N\to \R$ of $X:$
\begin{align*}
C(\mathbf{v})&=\E[\tau_{\mathbf{H}}X (\mathbf{t})\tau_{\mathbf{H}}X (\mathbf{t}+\mathbf{v})]=\\
\nonumber&=e^{-2H_1t_1}\cdots e^{-2H_N t_N} e^{-H_1v_1}\cdots e^{-H_N v_N} \E [X(e^{v_1+t_1},\ldots ,e^{v_N+t_N})X(e^{t_1},\ldots ,e^{t_N})]\\
&= K((e^{-v_1/2},\ldots ,e^{-v_N/2}) ,(e^{v_1/2},\ldots ,e^{v_N/2})), \mathbf{t},\mathbf{v} \in\R^N.
\end{align*}
The covariance function $K$ can be written as
\begin{equation}
\label{defLam1}
K(\mathbf{s},\mathbf{t})=\E X(\mathbf{s}) X(\mathbf{t})=\begin{cases} \prod_{k=1}^N (t_k s_k)^{H_k}C\left(\ln \frac{t_1}{s_1},\ldots,\ln \frac{t_N}{s_N}\right),& \mathbf{s,t}\in (0,+\infty)^N,\\
0, &\mathbf{s,t}\in \R_+^N \setminus (0,+\infty)^N.\end{cases}
\end{equation}

The Lamperti transformation $\tau_{\mathbf{H}} B^{\mathbf{H}}$ of the fractional Brownian sheet is a centered Gaussian strictly stationary random field with covariance function (see \cite{taqqu}):
\begin{align}
\label{rfbs}
C_{fBs}(\mathbf{v})&=\prod_{i=1}^N\left(\cosh(H_i v_i)-2^{(2 H_i -1)}\left|\sinh(v_i/2)\right|^{2 H_i}\right), \mathbf{v} \in \R^N.
\end{align}

We need to prove an auxiliary lemma.
\begin{lemma}
\label{lmm21} Let a centered strictly stationary Gaussian random field  $\{Z(\mathbf{v}),\mathbf{v}\in \R^N\}$ have a covariance function $C:\R^N\to \R,$ $\E Z^2(\mathbf{1})=1$ and $\{X(\mathbf{t}),\mathbf{t}\in \R_+^N\},$ $\mathbf{H}=(H_1,\ldots,H_N)\in (0,1)^N$ be the inverse Lamperti transformation $\tau_{\mathbf{H}}^{(-1)}Z,$
defined in \eqref{defLam}. Let $S$ be a subset of indices $\{1,\ldots,N\}$ and $\bar{S}=\{1,\ldots,N\}\setminus S$ be its complement set. If $X$ has mild stationary rectangular increments, then for all  $(v_1,\ldots,v_N)\in \R_+^N,$ such that $v_k\neq 0,$ $k \in \bar{S},$ the function $C$ satisfies
\begin{align}
 \label{leq13} &\prod_{k \in \bar{S}} \left(e^{\frac{v_k}{2}}-e^{-\frac{v_k}{2}}\right)^{2H_k}\\
\nonumber&=\sum_{\substack{i_k,j_k \in \{0,1\},k \in \bar{S}\\ i_k=j_k=0, k \in S}  }\left(\prod_{k \in \bar{S}} (-1)^{i_k+j_k}  e^{v_k H_k(1-i_k-j_k)}\right) C\left(v_1(i_1-j_1),\ldots,v_N(i_N-j_N)\right).
\end{align}
\end{lemma}
\begin{proof}
For an arbitrary point $(v_1,\ldots,v_N)\in \R_+^N,$ such that $v_k\neq 0,$ $k \in \bar{S},$ we consider the increment of $X$ on the rectangle $\prod_{k=1}^N[s_k,t_k],$ where $s_k >0, t_k:=s_k e^{v_k},$ if $k\in \bar{S},$ and $s_k:=0, t_k>0,$ if $k\in {S}.$ Denote $\mathbf{s}=(s_1,\ldots,s_N),\mathbf{t}=(t_1,\ldots,t_N).$ The increment $\Delta_{\mathbf{s}} X(\mathbf{t}),$ defined by \eqref{rec:incr}, has the form
\begin{align}
\nonumber
\Delta_{\mathbf{s}} X(\mathbf{t})&=\sum_{(i_1,\ldots,i_{N})\in \{0,1\}^N}(-1)^{i_1+\cdots + i_N} X(t_1-i_1 (t_1-s_1),\ldots,t_N-i_N (t_N-s_N)).
\end{align}
In the last sum, the terms corresponding to $i_k=1, k\in S$ equal 0 a.s. This follows from Proposition~\ref{prop1}, $(ii)$, because for $i_k=1$ the $k$th coordinate equals $t_k-i_k (t_k-s_k)=s_k=0$ if $k\in S.$ Therefore, \begin{align}
\label{leq01} \Delta_{\mathbf{s}} X(\mathbf{t})
 & = \sum_{\substack{i_k \in \{0,1\},k \in \bar{S}\\
i_k=0, k \in S}}(-1)^{i_1+\cdots + i_N} X(t_1-i_1 (t_1-s_1),\ldots,t_N-i_N (t_N-s_N)).
\end{align}

At the same time, Proposition~\ref{prop1}, $(iv)$ gives
\begin{equation*}\begin{gathered}\E[\Delta_{\mathbf{{s}}} X(\mathbf{{t}})]^2 =\prod_{k=1}^N (t_k-s_k)^{2 H_k}=\prod_{k\in S\cup \bar{S}} (t_k-s_k)^{2H_k}\\=\prod_{k\in S}t_k^{2H_k}  \prod_{k\in \bar{S}}(t_k-s_k)^{2H_k}.\end{gathered}\end{equation*} Further, from \eqref{leq01} we have the following equality of variances
\begin{align}
\label{leq11}& \prod_{k\in S} t_k^{2 H_k} \prod_{k \in \bar{S}} (t_k-s_k)^{2 H_k}= \E[\Delta_{\mathbf{s}} X(\mathbf{t})]^2=
\sum_{\substack{i_k, j_k \in \{0,1\}, k \in \bar{S}\\ i_k=j_k=0, k \in S}}(-1)^{i_1+j_1+\cdots + i_N+j_N}\\
\nonumber &\times \E X(t_1-i_1 (t_1-s_1),\ldots,t_N-i_N (t_N-s_N))X(t_1-j_1 (t_1-s_1),\ldots,t_N-j_N (t_N-s_N)).
\end{align}
Applying \eqref{defLam1}, we write the last equality in terms of covariance function $C:$
\begin{align}
\nonumber \prod_{k\in S} t_k^{2 H_k} \prod_{k \in \bar{S}} (t_k-s_k)^{2 H_k}&= \sum_{\substack{i_k,j_k \in \{0,1\}, k \in \bar{S}\\ i_k=j_k=0, k \in S}}(-1)^{i_1+j_1+\cdots + i_N+j_N}\\
\nonumber &\times \left(\prod_{k\in S} t_k^{2 H_k}\right)\left(\prod_{k\in \bar{S}} \left((t_k-i_k (t_k-s_k))(t_k-j_k (t_k-s_k))\right)^{H_k}\right) \\
 \label{leq11-1} &\times C\left(\ln \frac{t_N-i_N (t_N-s_N)}{t_N-j_N (t_N-s_N)},\ldots,\frac{t_N-i_N (t_N-s_N)}{t_N-j_N (t_N-s_N)}\right).
\end{align}
In term of \eqref{leq11-1}, the $k$th coordinate of function $C$ equals $0$ if $k\in S.$ Indeed, for $k\in S$
$i_k=j_k=0$ and $\ln \frac{t_k-i_k (t_k-s_k)}{t_k-j_k (t_k-s_k)}=\ln \frac{t_k}{t_k}=0.$

Now we recall that $t_k=s_k e^{v_k}, k\in \bar{S}.$ Then for $k\in \bar{S}$ the $k$th coordinate of function $C$ in term \eqref{leq11-1} equals
$$\ln \frac{e^{v_k}-i_k (e^{v_k}-1)}{e^{v_k}-j_k (e^{v_k}-1)}=\ln \frac{e^{v_k/2}-i_k (e^{v_k/2}-e^{-v_k/2})}{e^{v_k/2}-j_k (e^{v_k/2}-e^{-v_k/2})}.$$

For $k\in S$ we also rewrite the $k$th coordinate of  $C$ as $\ln \frac{e^{v_k/2}-i_k (e^{v_k/2}-e^{-v_k/2})}{e^{v_k/2}-j_k (e^{v_k/2}-e^{-v_k/2})},$ because this term equals 0 if $i_k=j_k=0.$ Hence, equality \eqref{leq11} is equivalent to
\begin{align}
\nonumber  &\left(\prod_{k\in S} t_k^{2 H_k}\right)\left( \prod_{k \in \bar{S}} s_k^{2H_k} e^{v_k H_k}\left(e^{\frac{v_k}{2}}-e^{-\frac{v_k}{2}}\right)^{2H_k}\right)=\sum_{\substack{i_k,j_k \in \{0,1\},k \in \bar{S}\\ i_k=j_k=0, k \in S}  }\\
\nonumber &\times\left(\prod_{k \in \bar{S}} (-1)^{i_k+j_k} s_k^{2H_k}e^{v_k H_k} \left(e^{\frac{v_k}{2}}-i_k (e^{\frac{v_k}{2}}-e^{-\frac{v_k}{2}})\right)^{H_k}\left(e^{\frac{v_k}{2}}-j_k (e^{\frac{v_k}{2}}-e^{-\frac{v_k}{2}})\right)^{H_k}\right) \\
\label{leq12} &\times \left(\prod_{k\in S} t_k^{2 H_k}\right) C\left(\ln \frac{e^{\frac{v_1}{2}}-i_1 (e^{\frac{v_1}{2}}-e^{-\frac{v_1}{2}})}{e^{\frac{v_1}{2}}-j_1 (e^{\frac{v_1}{2}}-e^{-\frac{v_1}{2}})},\ldots,\ln \frac{e^{\frac{v_N}{2}}-i_N (e^{\frac{v_N}{2}}-e^{-\frac{v_N}{2}})}{e^{\frac{v_N}{2}}-j_N (e^{\frac{v_N}{2}}-e^{-\frac{v_N}{2}})}\right).
\end{align}
After simplifications, we get that equality \eqref{leq13}  follows from \eqref{leq12}.
\end{proof}
The main result of this section is the following.
\begin{theorem}
\label{theor21}
Let a centered strictly stationary Gaussian random field  $\{Z(\mathbf{v}),\mathbf{v}\in \R^N\}$ have a covariance function $C:\R^N\to \R$ and $\E Z^2(\mathbf{1})=1.$ Then a Gaussian self-similar random field $\{X(\mathbf{t}),\mathbf{t}\in \R_+^N\}$ with index $(H_1,\ldots,H_N)\in (0,1)^N,$  defined in \eqref{defLam} as an inverse Lamperti transformation of $Z$, has mild stationary rectangular increments if and only if
\begin{align}
\label{leq00-1}
\sum_{\ve_k=\pm 1,1\leq k \leq N}C(\ve_1 v_1,\ldots, \ve_{N} v_{N})=2^{N}C_{fBs}(\mathbf{v}), \mathbf{v}\in \R^N.
\end{align}
\end{theorem}
\begin{proof} Let us prove the necessity. For an arbitrary point $(v_1,\ldots,v_N)\in \R_+^N,$ such that $v_k >0,$ $1\leq k \leq N,$ we apply Lemma \ref{lmm21}. From equality \eqref{leq13} we get
\begin{align}
\label{leq2} &\prod_{k=1}^N \left(e^{\frac{v_k}{2}}-e^{-\frac{v_k}{2}}\right)^{2H_k}=\\
\nonumber &\sum_{i_k,j_k\in \{0,1\}, 1\leq k \leq N }\left(\prod_{k=1}^N (-1)^{i_k+j_k}  e^{v_k H_k(1-i_k-j_k)}\right) C\left(v_1(i_1-j_1),\ldots,v_N(i_N-j_N)\right).
\end{align}

Denote the terms in the right hand side of  \eqref{leq2}
\begin{align*}
Q(\overline{i},\overline{j})&:=\left(\prod_{k=1}^N (-1)^{i_k+j_k}  e^{v_k H_k(1-i_k-j_k)}\right) C\left(v_1(i_1-j_1),\ldots,v_N(i_N-j_N)\right),\\
&\overline{i}=(i_1,\ldots,i_N)\in\{0,1\}^N,\overline{j}=(j_1,\ldots,j_N)\in\{0,1\}^N.
\end{align*}
Denote $A_k=\{(\overline{i},\overline{j})\in\{0,1\}^{2N},i_k=j_k\}.$  We recall the inclusion-exclusion principle for the indicator functions
$$\mathbb{I}\{\cdot \in \cup_{k=1}^N A_k\}=\sum_{I\subseteq \{1,\ldots,N\},I\neq \emptyset }(-1)^{|I|-1}\mathbb{I}\{\cdot \in \cap_{k\in I}A_k\}.$$
Applying the last formula, we write down the right hand side of equality \eqref{leq2} in the following form
\begin{align}
\nonumber  &\sum_{\substack{i_k, j_k\in\{0,1\}\\ 1\leq k \leq N}} Q(\overline{i},\overline{j})=  \sum_{\substack{i_k\in \{0,1\}\\ j_k=1-i_k, 1\leq k \leq N}} Q(\overline{i},\overline{j})+ \sum_{\substack{i_k, j_k\in\{0,1\}\\ 1\leq k \leq N}} Q(\overline{i},\overline{j})\mathbb{I}\{(\overline{i},\overline{j}) \in \cup_{k=1}^N A_k\}\\
\nonumber&=  \sum_{\substack{i_k\in \{0,1\}\\ j_k=1-i_k, 1\leq k \leq N}} Q(\overline{i},\overline{j})+ \sum_{\substack{i_k, j_k\in\{0,1\}\\ 1\leq k \leq N}} Q(\overline{i},\overline{j})\sum_{\substack{I\subseteq \{1,\ldots,N\}\\I\neq \emptyset} }(-1)^{|I|-1}\mathbb{I}\{(\overline{i},\overline{j}) \in \cap_{k\in I}A_k\}\\
\label{lmm:eq1}&=  \sum_{\substack{i_k\in \{0,1\}\\ j_k=1-i_k, 1\leq k \leq N}} Q(\overline{i},\overline{j})+ \sum_{\substack{I\subseteq \{1,\ldots,N\}\\I\neq \emptyset} }(-1)^{|I|-1} \sum_{\substack{i_k, j_k\in\{0,1\},1\leq k \leq N\\ i_k=j_k, k\in I}} Q(\overline{i},\overline{j}).
\end{align}
We write the last sum in terms of the function $C:$
\begin{align}
\nonumber&\sum_{\substack{i_k, j_k\in\{0,1\},1\leq k \leq N\\ i_k=j_k, k\in I}} Q(\overline{i},\overline{j})\\
\nonumber&= \sum_{\substack{i_k, j_k\in\{0,1\},1\leq k \leq N\\ i_k=j_k, k\in I}} \left(\prod_{k=1}^N (-1)^{i_k+j_k}  e^{v_k H_k(1-i_k-j_k)}\right) C\left(v_1(i_1-j_1),\ldots,v_N(i_N-j_N)\right)\\
\nonumber&=\sum_{ \substack{i_k, j_k\in\{0,1\},1\leq k \leq N\\ i_k=j_k, k\in I}  }\left(\prod_{k\not \in I} (-1)^{i_k+j_k}  e^{v_k H_k(1-i_k-j_k)}\right)C\left(v_1(i_1-j_1),\ldots,v_N(i_N-j_N)\right) \\
\nonumber&\times  \left(\prod_{k \in I} e^{v_k H_k(1-i_k-j_k)}\right) \\
\label{lmm:eq2}&=\sum_{ \substack{i_k, j_k\in\{0,1\},1\leq k \leq N\\ i_k=j_k=0, k\in I}  } \ \left(\prod_{k\not \in I} (-1)^{i_k+j_k}  e^{v_k H_k(1-i_k-j_k)}\right) C\left(v_1(i_1-j_1),\ldots,v_N(i_N-j_N)\right) \\
\nonumber&\times \prod_{k \in I} \left( e^{v_k H_k}+e^{-v_k H_k}\right).
\end{align}
We apply formula \eqref{leq13} for term \eqref{lmm:eq2}, where we set $S=I,$ and $\bar{S}=\{1,\ldots,N\}\setminus I.$ Therefore,
\begin{align}
\label{lmm:eq3}\sum_{\substack{i_k, j_k\in\{0,1\},1\leq k \leq N\\ i_k=j_k, k\in I}} Q(\overline{i},\overline{j})=\prod_{k\not \in I}\left(e^{\frac{v_k}{2}}-e^{-\frac{v_k}{2}}\right)^{2H_k} \prod_{k \in I} \left( e^{v_k H_k}+e^{-v_k H_k}\right).
\end{align}
Hence, using relations \eqref{lmm:eq1} and \eqref{lmm:eq3}, we get that equality \eqref{leq2} is equivalent to
\begin{align}
\nonumber &\prod_{k=1}^N \left(e^{\frac{v_k}{2}}-e^{-\frac{v_k}{2}}\right)^{2H_k}\\
\nonumber &=\sum_{\substack{i_k\in \{0,1\}\\ j_k=1-i_k, 1\leq k \leq N}} \left(\prod_{k=1}^N (-1)^{i_k+j_k}  e^{v_k H_k(1-i_k-j_k)}\right) C\left(v_1(i_1-j_1),\ldots,v_N(i_N-j_N)\right)\\
\nonumber &+ \sum_{\substack{I\subseteq \{1,\ldots,N\}\\I\neq \emptyset} }(-1)^{|I|-1} \prod_{k\not \in I}\left(e^{\frac{v_k}{2}}-e^{-\frac{v_k}{2}}\right)^{2H_k} \prod_{k \in I} \left( e^{v_k H_k}+e^{-v_k H_k}\right)\\
\nonumber &=(-1)^N\sum_{\substack{i_k\in \{0,1\}\\ j_k=1-i_k, 1\leq k \leq N}} C\left(v_1(i_1-j_1),\ldots,v_N(i_N-j_N)\right)\\
\label{leq22}&-(-1)^{N} \sum_{\substack{I\subseteq \{1,\ldots,N\}\\I\neq \emptyset} }\left(\prod_{k \not\in I} (-1)\left(e^{\frac{v_k}{2}}-e^{-\frac{v_k}{2}}\right)^{2H_k}\right) \prod_{k \in I} \left( e^{v_k H_k}+e^{-v_k H_k}\right).
\end{align}
From the last equality we have that
\begin{align}
\nonumber &\sum_{\substack{i_k\in \{0,1\}\\ j_k=1-i_k, 1\leq k \leq N}} C\left(v_1(i_1-j_1),\ldots,v_N(i_N-j_N)\right)=(-1)^N \prod_{k=1}^N \left(e^{\frac{v_k}{2}}-e^{-\frac{v_k}{2}}\right)^{2H_k}\\
\nonumber &+\sum_{\substack{I\subseteq \{1,\ldots,N\}\\I\neq \emptyset} }\left(\prod_{k \not\in I} (-1)\left(e^{\frac{v_k}{2}}-e^{-\frac{v_k}{2}}\right)^{2H_k}\right) \prod_{k \in I} \left( e^{v_k H_k}+e^{-v_k H_k}\right)\\
\label{leq23} = &\prod_{k=1}^N \left( \left( e^{v_k H_k}+e^{-v_k H_k}\right)-\left(e^{\frac{v_k}{2}}-e^{-\frac{v_k}{2}}\right)^{2H_k} \right)
= 2^N C_{fBs}(v_1,\ldots,v_N).
\end{align}

Similarly, we can show that equality \eqref{leq23} also holds true if $v_k=0$ for some  $k\in \{1,\ldots,N\}.$  Thus, we obtain that the covariance function of the field $\tau_\mathbf{H}X$ needs to satisfy \eqref{leq00-1}.

Now let us prove the sufficiency. Let  \eqref{leq00-1} hold  true, and let us  write this equality for an arbitrary point $(v_1,\ldots,v_N)\in \R_+^N,$ such that $v_k=0,$ $k\in S\subseteq\{1,\ldots,N\}.$ Denote $\bar{S}=\{1,\ldots,N\}\subset S.$ Then \eqref{leq00-1} rewrites
\begin{align}
\nonumber &\sum_{\substack{i_k\in \{0,1\}\\ j_k=1-i_k, 1\leq k \leq N}} C\left(v_1(i_1-j_1),\ldots,v_N(i_N-j_N)\right)=
\end{align}
\begin{align}
\nonumber = 2^{|S|}&\sum_{\substack{i_k\in \{0,1\},j_k=1-i_k, k\in \bar{S}\\i_k=j_k=0, k\in S}} C\left(v_1(i_1-j_1),\ldots,v_N(i_N-j_N)\right)\\
\label{thrm:eq1}&= 2^N 2^{|S|-N}\prod_{k\in \bar{S}}\left(  e^{v_k H_k}+e^{-v_k H_k}-\left(e^{\frac{v_k}{2}}-e^{-\frac{v_k}{2}}\right)^{2H_k} \right).
\end{align}
 In term \eqref{thrm:eq1}, it holds that $i_k=j_k$ for $k\in S$ and the $k$th coordinate of the function $C$ equals $v_k(i_k-j_k)=0,$ which does not depend on value of $v_k.$ The right hand side of \eqref{thrm:eq1} is also independent of  $v_k,k\in S,$ and, therefore, equality \eqref{thrm:eq1} holds true for any $(v_1,\ldots,v_N)\in \R_+^N.$

Now we prove that equality \eqref {leq2} is true. Its right hand side is equal to
\begin{align}
\nonumber&\sum_{\substack{i_k, j_k\in\{0,1\}\\ 1\leq k \leq N}} \left(\prod_{k=1}^N (-1)^{i_k+j_k}  e^{v_k H_k(1-i_k-j_k)}\right) C\left(v_1(i_1-j_1),\ldots,v_N(i_N-j_N)\right)\\
\nonumber&=  \sum_{S\subset\{1,\ldots,N\}}\sum_{\substack{i_k, j_k\in\{0,1\}\\i_k=j_k,k\in S\\
 j_k=1-i_k, k\in \bar{S}}}  \left(\prod_{k=1}^N (-1)^{i_k+j_k}  e^{v_k H_k(1-i_k-j_k)}\right) C\left(v_1(i_1-j_1),\ldots,v_N(i_N-j_N)\right)\\
\nonumber&=  \sum_{S\subset\{1,\ldots,N\}}\sum_{\substack{i_k, j_k\in\{0,1\}\\i_k=j_k,k\in S\\
 j_k=1-i_k, k\in \bar{S}}}   C\left(v_1(i_1-j_1),\ldots,v_N(i_N-j_N)\right)\\
\nonumber&\times \left(\prod_{k\in \bar{S}}^N (-1)^{i_k+j_k}  e^{v_k H_k(1-i_k-j_k)}\right)\left(\prod_{k\in S}^N (-1)^{i_k+j_k}  e^{v_k H_k(1-i_k-j_k)}\right)\\
\nonumber&=  \sum_{S\subset\{1,\ldots,N\}}\sum_{\substack{i_k, j_k\in\{0,1\}\\i_k=j_k=0,k\in S\\
 j_k=1-i_k, k\in \bar{S}}}   C\left(v_1(i_1-j_1),\ldots,v_N(i_N-j_N)\right) (-1)^{|\bar{S}|}\left(\prod_{k\in S}^N \left(e^{v_k H_k}+e^{v_k H_k}\right)\right)\\
\label{thrm:eq2} &=  \sum_{S\subset\{1,\ldots,N\}}\left(\prod_{k\in S}^N \left(e^{v_k H_k}+e^{v_k H_k}\right)\right) \sum_{\substack{i_k, j_k\in\{0,1\}\\i_k=j_k=0,k\in S\\
 j_k=1-i_k, k\in \bar{S}}}   C\left(v_1(i_1-j_1),\ldots,v_N(i_N-j_N)\right) (-1)^{|\bar{S}|}.
\end{align}
In the last sum, we apply \eqref{thrm:eq1}. Then the right hand side of  \eqref{thrm:eq2} equals
\begin{align}
\nonumber &\sum_{S\subset\{1,\ldots,N\}}\prod_{k\in S}^N\left(e^{v_k H_k}+e^{v_k H_k}\right) \sum_{\substack{i_k, j_k\in\{0,1\}\\i_k=j_k=0,k\in S\\
 j_k=1-i_k, k\in \bar{S}}}  \prod_{k\in \bar{S}}(-1)\left(  e^{v_k H_k}+e^{-v_k H_k}-\left(e^{\frac{v_k}{2}}-e^{-\frac{v_k}{2}}\right)^{2H_k} \right)\\
&=\prod_{k=1}^N\left(e^{v_k H_k}+e^{v_k H_k} -\left(  e^{v_k H_k}+e^{-v_k H_k}-\left(e^{\frac{v_k}{2}}-e^{-\frac{v_k}{2}}\right)^{2H_k} \right)\right).
\end{align}
Equality \eqref{leq2} follows from the last assertion.

Let $\mathbf{s}=(s_1,\ldots,s_N)\in \R_+^N$, $\mathbf{t}=(t_1,\ldots,t_N)\in \R_+^N$ be such that  $0<s_k\leq t_k,1\leq k \leq N.$ Using \eqref{defLam1}, equality \eqref{leq2} with $v_k=\ln(t_k/s_k),1\leq k \leq N$ is equivalent to \linebreak
$\prod_{k=1}^N |t_k-s_k|^{2 H_k}=\E[\Delta_{\mathbf{s}} X(\mathbf{t})]^2.$ Since $\E[\Delta_{\mathbf{0}} X(\mathbf{t}-\mathbf{s})]^2=\E [X(\mathbf{t}-\mathbf{s})]^2=\prod_{k=1}^N |t_k-s_k|^{2 H_k},$ we have $\E[\Delta_{\mathbf{s}} X(\mathbf{t})]^2=\E[\Delta_{\mathbf{0}} X(\mathbf{t}-\mathbf{s})]^2.$ The fact that the distributions of the increments are invariant w.r.t. translations, follows from the last identity and the fact that the increments of the field $X$ are centered Gaussian random variables.
\end{proof}

In the case $N=2,$ equality  \eqref{leq00-1} has the form
\begin{equation}
\label{eq1}
C( v_1,v_2)+C( -v_1,v_2)=2C_{fBs}(\mathbf{v}).
\end{equation}

In the paper \cite{mak_mish}, a certain class of covariance functions satisfying \eqref{eq1} is given by
%
\begin{align}
\label{thrm_cov} {C}_\theta(\mathbf{v})=&C_{fBs}(\mathbf{v}) \left(1+\theta e^{-H_1 |v_1|-H_2|v_2|}\sinh \left(H_1 v_1\right)\sinh \left(H_2 v_2\right)\right), \mathbf{v}\in \R^2,
\end{align}
where $0<H_1<1,0<H_2<1,$ $\theta\in\R$ are some numbers.

\section{Spectral representation of the fractional Brownian motion}
By Bochner's theorem, a continuous at the origin $0\in \R^N$ covariance function $C$ of a strictly stationary random field can be represented as a characteristic function of a finite spectral measure. Assume that this spectral measure has a spectral density $f: \R^N \rightarrow \R_+,$ i.e.,
\begin{equation}
\label{fur}
C(\mathbf{v})=\int_{\R^N}e^{i <\mathbf{x},\mathbf{v}>}f(\mathbf{x})d\mathbf{x}, \mathbf{v}\in \R^N.
\end{equation}

Let  $\{Z(\mathbf{t}),\mathbf{t}\in \R^N\}$ be a centered Gaussian strictly stationary random field with spectral density $f.$ Then  $Z$ has the representation
\[Z(\mathbf{s})=\int_{\R^N}e^{i<\mathbf{s},\mathbf{x}>}\sqrt{f(\mathbf{x})}M(d\mathbf{x}), \mathbf{s}\in \R^N, \] where $M$ is a centered Gaussian random measure  with control Lebesgue measure.  Then a random field  $\{X(\mathbf{t}),\mathbf{t}\in \R_+^N\}$, defined as inverse Lamperti transform \eqref{defLam},  has the following spectral representation at point $\mathbf{t}\in (0,+\infty)^N$
\begin{align*}
X(\mathbf{t})&=t_1^{H_1}\cdots t_N^{H_N}Z(\ln t_1,\ldots, \ln t_N)\\
&=t_1^{H_1}\cdots t_N^{H_N}\int_{\R^N}\exp\left(i\sum_{k=1}^N  x_k \ln t_k\right)\sqrt{f(\mathbf{x})}M(d\mathbf{x}).
\end{align*}

First we find the spectral density of $C_{fBs}.$ Since it is the coordinate-wise product of covariance functions, then its spectral density is coordinate-wise product too. Thus, it is sufficient to consider only the case $N=1.$ We also use this result to obtain a new spectral representation of the fractional Brownian motion.
\begin{theorem}
Let $\{B^H(t),t\in \R_+^N\}$ be the fractional Brownian motion with Hurst index $H\in(0,1),$ and  $C_{fBs}$ be the covariance function \eqref{rfbs} of the Lamperti transformation $\tau_H B^{H}.$ Then  $C_{fBs}$ has the following spectral density
\begin{align}
\label{defg}
g(x)=\frac{1}{2\pi}\frac{2H}{H^2+x^2}\frac{\pi \Gamma(2H)}{|\Gamma(H+ix)|^2}\frac{\sin(\pi H) \cosh(\pi x)}{\sin^2(\pi H) \cosh^2(\pi x)+\cos^2(\pi H)\sinh^2(\pi x)},x\in \R,
\end{align}
 where $\Gamma(\cdot)$ is the gamma function of complex argument. The fractional Brownian motion $B^H$ has the following representation
\begin{align}
\label{fbm:repr}
B^H(t)=\int_{\R}\frac{t^H e^{i x \ln t}}{|\Gamma(H+ix)|\cdot |\sin (\pi(H+ix))|}\left(\frac{H \Gamma(2H) \sin(\pi H) \cosh(\pi x)}{H^2+x^2}\right)^{1/2}M(dx).
\end{align}
The spectral density of Lamperti transformation  of the Brownian motion equals
\begin{equation}
\label{defgs}
g_W(x)=
\frac{1}{2\pi}\frac{1}{(1/2)^2+x^2},x\in \R.
\end{equation}
\end{theorem}
\begin{proof}
We look for the density $g(x),x\in \R$ in the form of the inverse Fourier transform of $C_{fBs}.$ In the paper \cite{mak_mish} it is showed that $C_{fBs}$ is integrable.  Indeed,
\begin{align*}
&g(x)=\frac{1}{2\pi}\int_{\R}e^{-i x v}C_{fBs}(v)d v\\
&=\frac{1}{2\pi} \int_{\R}e^{-i x  v }\left(\cosh(H  v )-2^{(2 H -1)}\left|\sinh(v/2)\right|^{2 H}\right)d v\\
\label{Idef}&=\frac{1}{4\pi} \int_{\R}e^{-i x v}\left(e^{H v}+e^{-H v}-\left|e^{v/2}-e^{-v/2}\right|^{2 H}\right)d v=:\frac{ I_{H}(x)}{4\pi}, x\in \R.
\end{align*}
Let us compute $I_{H}.$ Since the integrand in $I_{H}$ excluding $e^{-i x v}$ is an even function of $v,$ then $I_{H}$ is a real-valued even function. Therefore,
\begin{align*}
I_{H}(x)&=\int_{\R_+}(e^{-i x v}+e^{i x v})\left(e^{-H v}+e^{H v}(1-(1-e^{-v})^{2 H})\right)d v =|e^{-v}=t|\\
&=\int_{0}^{1}(t^{ix}+t^{-ix})(t^{H}+t^{-H}-t^{-H}(1-t)^{2H})\frac{dt}{t}\\
&=\int_{0}^{1}t^{H-1+ix}dt+\int_{0}^{1}t^{H-1-ix}dt+\int_{0}^{1}(t^{ix}+t^{-ix}) (1-(1-t)^{2H})t^{-H-1}dt\\
&=\frac{1}{H+ix}+\frac{1}{H-ix}+\int_{0}^{1} (1-(1-t)^{2H}) d \left(\frac{t^{-H+ix}}{-H+ix}+\frac{t^{-H-ix}}{-H-ix}\right)\\
&=\frac{1}{H+ix}+\frac{1}{H-ix}+\lim_{t\to 1-}(1-(1-t)^{2H}) \left(\frac{t^{-H+ix}}{-H+ix}+\frac{t^{-H-ix}}{-H-ix}\right)\\
&-\lim_{t\to 0+}(1-(1-t)^{2H}) \left(\frac{t^{-H+ix}}{-H+ix}+\frac{t^{-H-ix}}{-H-ix}\right)\\
&-\int_{0}^{1} 2H(1-t)^{2H-1} \left(\frac{t^{-H+ix}}{-H+ix}+\frac{t^{-H-ix}}{-H-ix}\right)dt\\
&=-\lim_{t\to 0+}2H \left(\frac{t^{1-H+ix}}{-H+ix}+\frac{t^{1-H-ix}}{-H-ix}\right)\\
&+\frac{2H}{H-ix}B(2H,-H+1+ix)+\frac{2H}{H+ix}B(2H,-H+1-ix)\\
&=\frac{2H}{H-ix}\frac{\Gamma(2H)\Gamma(-H+1+ix)}{\Gamma(H+1+ix)}+\frac{2H}{H+ix}\frac{\Gamma(2H)\Gamma(-H+1-ix)}{\Gamma(H+1-ix)}.
\end{align*}
We recall some properties of the gamma function (see \cite{adam}):
\[\Gamma(1+z)=z \Gamma(z), \quad \Gamma(1-z)\Gamma(z) = \frac{\pi}{\sin{\pi z}},\quad \Gamma(z)\Gamma(\bar{z})=\left|\Gamma(z)\right|^2, z \in \mathbb{C},\] where
\[\sin(u+i v)= \sin u \cosh v + i \cos u \sinh v, \quad u,v \in \R.\]
Applying them, we get
\begin{align}
\nonumber
I_{H}(x)&=\frac{2H \Gamma(2H)}{H^2+x^2}\left(\frac{\Gamma(-H+1+ix)}{\Gamma(H+ix)}+\frac{\Gamma(-H+1-ix)}{\Gamma(H-ix)}\right)\\
\nonumber&=\frac{2H \Gamma(2H)}{H^2+x^2}\frac{\Gamma(1-(H-ix))\Gamma(H-ix)+\Gamma(1-(H+ix))\Gamma(H+ix)}{\Gamma(H+ix)\Gamma(H-ix)}\\
\nonumber&=\frac{2H\Gamma(2H)}{(H^2+x^2)|\Gamma(H+ix)|^2}\left(\frac{\pi}{\sin(\pi (H-ix))}+\frac{\pi}{\sin( \pi (H+ix))}\right)\\
\label{eqI}&=\frac{2H}{(H^2+x^2)}\frac{2 \pi \Gamma(2H)}{|\Gamma(H+ix)|^2}\frac{\sin(\pi H) \cosh(\pi x)}{\sin^2(\pi H) \cosh^2(\pi x)+\cos^2(\pi H)\sinh^2(\pi x)}.
\end{align}
In the case $H=1/2,$ formula \eqref{eqI} is simplified to
\begin{align}
I_{1/2}(x)=\frac{2}{(1/2)^2+x^2},x\in \R.
\end{align}
\end{proof}
For multidimensional case we have the immediate corollary.
\begin{corollary}
Let $\{B^\mathbf{H}(\mathbf{t}),\mathbf{t}\in \R_+^N\}$ be the fractional Brownian sheet with Hurst index $\mathbf{H}=(H_1,\ldots,H_N)\in(0,1)^N,$ and  $C_{fBs}$ be the covariance function \eqref{rfbs} of the Lamperti transformation $\tau_{\mathbf{H}}B^{\mathbf{H}}.$ Then $C_{fBs}$ has the following spectral density
\begin{align}
\label{defg2}
g_N(\mathbf{x})=\frac{1}{2\pi}\prod_{k=1}^N\frac{2 H_k}{H_k^2+x_k^2}\frac{\pi \Gamma(2H_k)}{|\Gamma(H_k+ix_k)|^2}\frac{\sin(\pi H_k) \cosh(\pi x_k)}{\sin^2(\pi H_k) \cosh^2(\pi x_k)+\cos^2(\pi H_k)\sinh^2(\pi x_k)},
\end{align}
$\mathbf{x}\in \R^N.$
The spectral density of Lamperti transformation  $\tau_{\mathbf{H}}W$ of the Brownian sheet equals
\begin{equation}
\label{defgs2}
g_{W,N}(\mathbf{x})=
\frac{1}{2\pi}\prod_{k=1}^N\frac{1}{(1/2)^2+x_k^2},\mathbf{x}\in \R^N.
\end{equation}
\end{corollary}

Now we will characterize all covariance functions for which \eqref{leq00-1} is true and rewrite equality \eqref{leq00-1} in terms of spectral densities.
\begin{theorem} Let a real-valued, centered, strictly stationary Gaussian random field $\{Z(\mathbf{t}),\mathbf{t}\in \R^N\}$ has the spectral density $f:\R^N\to \R_+.$ Then a Gaussian self-similar random field $\{X(\mathbf{t}),\mathbf{t}\in \R_+^N\}$ with index $\mathbf{H}=(H_1,\ldots,H_N)\in (0,1)^N,$  defined in \eqref{defLam} as inverse Lamperti transformation, has mild stationary rectangular increments if and only if
\begin{equation}
\label{spektreq1}
\sum_{\ve_1=\pm 1,\ldots,\ve_{N}= \pm 1}f(\ve_1 v_1,\ldots, \ve_{N} v_{N})=2^{N} g_N(\mathbf{v}), \mathbf{v}\in \R^N,
\end{equation}
where $g_N$ is the spectral density \eqref{defg2}.
\end{theorem}
\begin{proof}
Consider the left hand side of \eqref{leq00-1}:
\begin{align*}
&\sum_{\ve_1=\pm 1,\ldots,\ve_{N}= \pm 1}C(\ve_1 v_1,\ldots, \ve_{N} v_{N})=\sum_{\ve_1=\pm 1,\ldots,\ve_{N}= \pm 1}\int_{\R^N} \exp\left(i\sum_{k=1}^N \ve_k v_k x_k\right)f(\mathbf{x})d\mathbf{x}\\
&=\sum_{\ve_1=\pm 1,\ldots,\ve_{N}= \pm 1}\int_{\R^N} \exp\left(i\sum_{k=1}^N v_k y_k\right)f\left(\frac{y_1}{\ve_1},\ldots, \frac{y_N}{\ve_N}\right)d\mathbf{y}\\
&=\int_{\R^N} \exp\left(i\sum_{k=1}^N v_k y_k\right)\sum_{\ve_1=\pm 1,\ldots,\ve_{N}= \pm 1} f\left(\ve_1 y_1,\ldots, \ve_N y_N\right)d\mathbf{y}.
\end{align*}
The right hand side of \eqref{leq00-1} has the same representation as the Fourier transform of $g_N$ and, therefore, we have \eqref{spektreq1}.
\end{proof}

\begin{corollary}
In the case $N=2,$ the function $f$ is the spectral density satisfying \eqref{spektreq1} if
\end{corollary}
\begin{enumerate}
\item $f\in L_1(\R^2)$ and $f\geq 0,$
\item $f$ is symmetric around $(0,0),$
\item \(f(-x_1,x_2)=f(x_1,-x_2)=2 g(x_1,x_2) -f(x_1,x_2), \text{ for all }\mathbf{x}\in \R_+^2,\)
\end{enumerate}
\begin{proof}
From the fact that the spectral density $f$ is symmetric with respect to $\mathbf{0},$ it follows that the corresponding covariance function $C$ is real-valued. Therefore, \eqref{spektreq1} has the form $f(x_1,x_2)+f(x_1,-x_2)=2 g(x_1,x_2),$ which is covered by condition 3.
\end{proof}
\begin{remark}
We can replace the condition $f(x_1,x_2)\geq 0, (x_1,x_2)\in \R^2$ by  $0 \leq f(x_1,x_2)\leq  2 g(x_1,x_2),$ for all $  \mathbf{x}\in \R_+^2.$
\end{remark}

\section{Gaussian self-similar random fields with $\mathbf{H}=(0.5,0.5)$}

In this section we consider Gaussian self-similar random fields from $\mathcal{C}^{\mathbf{H},2}_M, N=2$  with the index $ \mathbf{H} = (0.5,0.5)$ and we construct an example of the field from $\mathcal{C}^{\mathbf{H},2}_M \setminus \mathcal{C}^{\mathbf{H},2}_S,$ which has no independent increments.

\begin{remark}
  For the case $H_1\neq 1/2$ or $H_2\neq 1/2$ the increments of fractional Brownian sheet are not independent. For example, we consider the fractional Brownian sheet $\{B^{\mathbf{H}}(\mathbf{t}), \mathbf{t}\in \R_+^2\}$ with Hurst index $\mathbf{H}=(H_1,H_2)\in (0,1)^2, H_1\neq 1/2.$ For a fixed  $t_2>0$ we consider the process $B_1=\{B^{\mathbf{H}}(t_1,t_2),t_1\in \R_+\}.$ The Gaussian process $B_1$ is self-similar with increments of the form $B_1(t_1+u)-B_1(t_1)=\Delta_{(t_1,t_2)}B^{\mathbf{H}}(t_1+u,t_2).$  Therefore, $B_1$ has stationary rectangular increments. Hence, $B_1$ is the fractional Brownian motion with $\E B_1^2(1)=t_2^{2H_2}<+\infty$ and Hurst index $H=H_1\neq 1/2.$
It is known that the fractional Brownian motion has independent increments only in the case of $H = 1/2.$ Therefore, the increments of the process $B_1$ are not independent, and consequently, the rectangular increments of $B^{\mathbf{H}}$ are not independent too.
\end{remark}
Let  $\{Y_H(\mathbf{t}),\mathbf{t}\in\R^2_+\}$ be a centered Gaussian random field with covariance function given by \eqref{defLam1} $\E Y_H(\mathbf{t})Y_H(\mathbf{s})=\sqrt{t_1 s_1 t_2 s_2}C_\theta\left(\ln \frac{t_1}{s_1},\ln \frac{t_2}{s_2}\right),\mathbf{t},\mathbf{s}\in \R_+^2,$ where $C_\theta$ is the covariance function \eqref{thrm_cov}.
We write down it explicitly.
\begin{align}
\nonumber&\E Y_H(\mathbf{t})Y_H(\mathbf{s})=\sqrt{t_1 s_1 t_2 s_2}C_{fBs}\left(\ln \frac{t_1}{s_1},\ln \frac{t_2}{s_2}\right)\\
\nonumber&\times\left(1+\frac{\theta}{4}\left(\frac{s_1 \wedge t_1}{s_1 \vee t_1}\right)^{H_1}\left(\frac{s_2 \wedge t_2}{s_2 \vee t_2}\right)^{H_2} \left(\frac{t_1^{H_1}}{s_1^{H_1}}-\frac{s_1^{H_1}}{t_1^{H_1}}\right) \left(\frac{t_2^{H_2}}{s_2^{H_2}}-\frac{s_2^{H_2}}{t_2^{H_2}}\right) \right)\\
\label{bs}&=\frac{1}{4}\prod_{i=1}^{2}
\bigl(t_i^{2H_i}+s_i^{2H_i}-|t_i-s_i|^{2H_i}
\bigr) \left(1+\frac{\theta}{4} \frac{(t_1^{2 H_1}-s_1^{2 H_1})}{(s_1 \vee t_1)^{2 H_1}}\frac{(t_2^{2 H_2}-s_2^{2 H_2})}{(s_2 \vee t_2)^{2 H_2}} \right).
\end{align}
Then $Y_H\in \mathcal{C}^{\mathbf{H},2}_M.$ Now let $(H_1,H_2)= (0.5,0.5)$ and  consider a centered Gaussian random field $\{Y_{1/2}(\mathbf{t}),\mathbf{t}\in\R^2_+\}$ with covariance function
\begin{align}
\label{cov_nwf}
\E Y_{1/2}(\mathbf{t})Y_{1/2}(\mathbf{s})=(t_1\wedge s_1)(t_2\wedge s_2)\left(1+\frac{\theta}{4} \frac{t_1-s_1}{t_1 \vee s_1}\frac{t_2-s_2}{t_2 \vee s_2} \right), \mathbf{s},\mathbf{t}\in \R_+^2,
\end{align}
which is the version of the right hand side of \eqref{bs} in the case  $\mathbf{H}= (0.5,0.5).$
From \cite{mak_mish} follows that $Y_{H/2}$ is self-similar and has mild stationary rectangular increments, i.e., $Y_{1/2}\in \mathcal{C}^{\mathbf{H},2}_M$. To show that rectangular increments of $Y_{1/2}$ are not wide-sense stationary we write down their covariance function. Let  $\mathbf{t},\mathbf{s}, \mathbf{h}\in \R_+^2$ and for simplicity we assume that $t_1>s_1, t_2>s_2.$ Then from \eqref{rec:incr} we have
\begin{align}
\notag&\E \Delta_{\mathbf{h}} Y_{1/2}(\mathbf{t}+\mathbf{h})\Delta_{\mathbf{h}} Y_{1/2}(\mathbf{s}+\mathbf{h})=\sum_{\substack{i_1,i_{2}\in \{0,1\}\\ j_1,j_{2}\in \{0,1\}}}(-1)^{i_1+i_2+j_1+j_2}\\
\label{last:eq1}&\times \E Y_{1/2}(t_1+h_1-i_1t_1,t_2+h_2-i_2t_2) Y_{1/2}(s_1+h_1-j_1s_1,s_2+h_2-j_2s_2).
\end{align}
After series of simplifications we obtain that \eqref{last:eq1} equals
\begin{equation}
\label{last:eq2}
    s_1 s_2 \left(1+\frac{\theta}{4} \frac{(2h_1+s_1)(2h_2+s_2)(t_1-s_1)(t_2-s_2)}{(h_1+s_1)(h_2+s_2)(h_1+t_1)(h_2+t_2)}\right).
\end{equation}
Hence, we see that \eqref{last:eq2} depends on $h_1$ and $h_2,$ but $\E  |\Delta_{\mathbf{h}}Y_{1/2}(\mathbf{s}+\mathbf{h})|^2$ does not. This means that increments of $Y_{1/2}$  does not possesses wide-sense stationary rectangular increments, i.e., $Y_{1/2}\in \mathcal{C}^{\mathbf{H},2}_M \setminus \mathcal{C}^{\mathbf{H},2}_S.$

Let us check  whether $Y_{1/2}$ has  independent rectangular  increments. For $t_1>0,t_2>0$ we consider $\Delta_{\mathbf{0}}Y_{1/2}(t_1,t_2)=Y_{1/2}(t_1,t_2)$ and $\Delta_{(t_1,0)}Y_{1/2}(2t_1,2t_2)=Y_{1/2}(2t_1,2t_2)-Y_{1/2}(t_1,2t_2).$
It follows from \eqref{cov_nwf} that the covariance of these increments equals
\begin{align*}
\E\Delta_{\mathbf{0}}Y_{1/2}(t_1,t_2)\Delta_{(t_1,0)}Y_{1/2}(2t_1,2t_2)&=\E Y_{1/2}(t_1,t_2) Y_{1/2}(2t_1,2t_2)\\
-\E Y_{1/2}(t_1,t_2) Y_{1/2}(t_1,2t_2)&=t_1 t_2 \left(1+\frac{\theta}{4}\frac{t_1}{2 t_1}\frac{t_2}{2 t_2}\right)-t_1 t_2 \left(1+\frac{\theta}{4}\frac{0}{t_1}\frac{t_2}{2 t_2}\right)\\
&=\frac{\theta}{16}t_1t_2>0, \mbox{ if } \theta>0.
\end{align*}
Thus, these non-intersecting increments are not independent, in contrast to the Brownian sheet, which has independent increments. Thus, we have the following statement.
\begin{proposition}
\label{prop_wf_3} The Gaussian self-similar field $Y_{1/2}$ with index $\mathbf{H} = (0.5,0.5)$ and the covariance function \eqref{cov_nwf}, $\theta \neq 0$ belongs to $ \mathcal{C}^{\mathbf{H},2}_M \setminus \mathcal{C}^{\mathbf{H},2}_S$ and the rectangular increments of $Y_{1/2}$ are not independent.
\end{proposition}

We can provide the similar result for the class $ \mathcal{C}^{\mathbf{H},2}_W.$

\begin{proposition}
\label{prop_wf_4} Let a Gaussian random  field $Z_{1/2}$ with  has the covariance function $\eqref{rem2:eq},\gamma\in (0,1]$ then $Z_{1/2}$ is self-similar with index $\mathbf{H} = (0.5,0.5),$ belongs to $\mathcal{C}^{\mathbf{H},2}_W$ but the rectangular increments of $Z_{1/2}$ are not independent.
\end{proposition}
\begin{proof}
For $t_1>0,t_2>0$ we consider increments $\Delta_{\mathbf{0}}Z_{1/2}(t_1,t_2)$ and $\Delta_{(t_1,0)}Z_{1/2}(2t_1,2t_2)$ on non-intersecting rectangles.
It follows from \eqref{rem2:eq} that the covariance of these increments equals
\begin{align*}
&\E\Delta_{\mathbf{0}}Z_{1/2}(t_1,t_2)\Delta_{(t_1,0)}Z_{1/2}(2t_1,2t_2)=-\E Z_{1/2}(t_1,t_2) Z_{1/2}(t_1,2t_2) \\
&+\E Z_{1/2}(t_1,t_2) Z_{1/2}(2t_1,2t_2)=-t_1 t_2+t_1 t_2\\   &+\frac{\gamma}{\pi^2} \left(2 t_1 \log(2 t_1) - t_1 \log t_1 - ( 2t_1-t_1)\log |2t_1-t_1|\right)\\
&\times \left(2 t_1 \log(2 t_2) - t_2 \log t_2 - ( 2t_2-t_2)\log |2t_2-t_2|\right)=\frac{ 4 \gamma (\log 2)^2}{\pi^2}t_1t_2>0
, \mbox{ if } \gamma>0.
\end{align*}
\end{proof}

\section{Appendix}

\begin{proof}[Proof of Theorem \ref{t1}]
Let us prove the self-similarity of $V.$ Take arbitrary $a_1>0,a_2>0.$ Then we have the following relations for finite dimensional distributions.
\begin{align*}
&V_{H_1,H_2}(a_1 t_1,a_2 t_2) \stackrel{d}{=} \lim_{n\to \infty} \frac{L_1(r_{1,n})L_2(r_{2,n})}{r_{1,n}^{H_1}r_{2,n}^{H_2}}\sum_{\substack{k_1\in [0,a_1 t_1 r_{1,n}]\cap \Z \\ k_2\in [0,a_2 t_2 r_{2,n}]\cap \Z}}Y(k_1,k_2)\\
&=a_1^{H_1}a_2^{H_2} \lim_{n\to \infty} \frac{L_1(r_{1,n})}{L_1(a_1r_{1,n})}\frac{L_2(r_{2,n})}{L_2(a_2r_{2,n})}\frac{L_1(a_1r_{1,n})L_2(a_2r_{2,n})}{(a_1r_{1,n})^{H_1}(a_2 r_{2,n})^{H_2}}\sum_{\substack{k_1\in [0,t_1 a_1  r_{1,n}]\cap \Z \\ k_2\in [0, t_2 a_2 r_{2,n}]\cap \Z}}Y(k_1,k_2)\\
&\stackrel{d}{=}  a_1^{H_1}a_2^{H_2} V_{H_1,H_2}( t_1, t_2).
\end{align*}
The strict stationarity of rectangular increments of $V$ follows from strict stationarity of $Y.$ Indeed, 
\begin{align*}
&\Delta_{\mathbf{h}} V_{H_1,H_2}(\mathbf{u}+\mathbf{h})=
V_{H_1,H_2}(u_1+h_1,u_2+h_2)\\
&-V_{H_1,H_2}(h_1,u_2+h_2)-V_{H_1,H_2}(u_1+h_1,h_2)+V_{H_1,H_2}(h_1,h_2)\\
&\stackrel{d}{=}\lim_{n\to \infty}\frac{L_1(r_{1,n})L_2(r_{2,n})}{r_{1,n}^{H_1}r_{2,n}^{H_2}}\left(\sum_{\substack{k_1\in [0,(u_1+h_1) r_{1,n}]\cap \Z \\ k_2\in [0,(u_2+h_2) r_{2,n}]\cap \Z}}Y(k_1,k_2)-\sum_{\substack{k_1\in [0,(u_1+h_1) r_{1,n}]\cap \Z \\ k_2\in [0,h_2 r_{2,n}]\cap \Z}}Y(k_1,k_2)\right.\\
&\left.-\sum_{\substack{k_1\in [0,(u_1+h_1) r_{1,n}]\cap \Z \\ k_2\in [0,h_2 r_{2,n}]\cap \Z}}Y(k_1,k_2)+\sum_{\substack{k_1\in [0,h_1 r_{1,n}]\cap \Z \\ k_2\in [0,h_2 r_{2,n}]\cap \Z}}Y(k_1,k_2)\right)\\
&=\lim_{n\to \infty}\frac{L_1(r_{1,n})L_2(r_{2,n})}{r_{1,n}^{H_1}r_{2,n}^{H_2}}\sum_{\substack{k_1\in [h_1,(u_1+h_1) r_{1,n}]\cap \Z \\ k_2\in [h_2,(u_2+h_2) r_{2,n}]\cap \Z}}Y(k_1,k_2)\\
&\stackrel{d}{=} \lim_{n\to \infty} \frac{L_1(r_{1,n})L_2(r_{2,n})}{r_{1,n}^{H_1}r_{2,n}^{H_2}}\sum_{\substack{k_1\in [0,u_1 r_{1,n}]\cap \Z \\ k_2\in [0,u_2 r_{2,n}]\cap \Z}}Y(k_1,k_2) =\Delta_{\mathbf{0}} V_{H_1,H_2}(\mathbf{u}).
\end{align*}

\end{proof}

Here we state and prove some auxiliary lemmas.
Firstly, let us recall several defined integrals, which can be found for example in \cite[Relations 3.761, 3.784, 3.823]{reznik}
\begin{lemma}
\label{lemma0}
\begin{align}
\label{lemma0:eq1}\int_{\R_+}\frac{\cos (|a|  x)-\cos (|b|  x)}{x}dx&=\log |a| - \log {|b|},\\
\label{lemma0:eq2}\int_{\R_+}\frac{\cos x-1}{x^{\alpha}}dx&=\Gamma(1-\alpha)\sin\left(\frac{\pi}{2}\alpha\right), \text{ if }\alpha\in (1,2),\\
\label{lemma0:eq3}\int_{\R_+}\frac{\cos x}{x^{\alpha}}dx&=\Gamma(1-\alpha)\sin\left(\frac{\pi}{2}\alpha\right), \text{ if }\alpha\in (0,1),\\
\label{lemma0:eq4}\int_{\R_+}\frac{\sin x}{x^{\alpha}}dx&=\Gamma(1-\alpha)\cos\left(\frac{\pi}{2}\alpha\right), \text{ if }\alpha\in (0,2).
\end{align}
\end{lemma}

\begin{lemma}
\label{lemma1}
Let $H\in (0,1),H\neq \frac12$ and $s,t\in \R,$ then
\begin{align}
\nonumber
&\int_{\R_+}\frac{\left(e^{i t y}-1\right)\left(e^{-i s y}-1\right)}{y^{2H+1}}dy=\frac{\pi}{\Gamma(1+2H)\sin (2\pi H)}\\
\label{lemma1:eq}&\times \left(e^{-i \pi H \sign t}|t|^{2H} + e^{i \pi H \sign s}|s|^{2H}-e^{-i \pi H \sign (t-s)}|t-s|^{2H}\right).
\end{align}
\end{lemma}
\begin{proof}
Consider the real part of \eqref{lemma1:eq}.
\begin{align}
\notag    \Re \int_{\R_+}\frac{\left(e^{i t y}-1\right)\left(e^{-i s y}-1\right)}{y^{2H+1}}dy&=\int_{\R_+}\frac{\cos((t-s)y)-1}{y^{2H+1}}dy\\
\notag    & -\int_{\R_+}\frac{\cos(s y)-1}{y^{2H+1}}dy - \int_{\R_+}\frac{\cos(t y)-1}{y^{2H+1}}dy\\
\notag    &=\left(|t-s|^{2H} -|s|^{2H} -|t|^{2H}\right)\int_{\R_+}\frac{\cos z -1}{z^{2H+1}}dy\\
\label{lemma1:eq7}    &\stackrel{eq. \eqref{lemma0:eq2}}{=}-\left(|t|^{2H} +|s|^{2H} -|t-s|^{2H}\right)\Gamma(-2H) \cos( \pi H).
\end{align}
Consider the imaginary part of \eqref{lemma1:eq} in the case $H\in (0,\frac12).$
\begin{align}
\notag    &\Im \int_{\R_+}\frac{\left(e^{i t y}-1\right)\left(e^{-i s y}-1\right)}{y^{2H+1}}dy\\
\notag &=\int_{\R_+}\frac{\sin((t-s)y)}{y^{2H+1}}dy + \int_{\R_+}\frac{\sin(sy)}{y^{2H+1}}dy -\int_{\R_+}\frac{\sin(ty)}{y^{2H+1}}dy \\
\notag    &=\left(\sign(t-s)|t-s|^{2H} +\sign{s} |s|^{2H} -\sign t |t|^{2H}\right)\int_{\R_+}\frac{\sin z}{z^{2H+1}}dy\\
 \label{lemma1:eq6}   &\stackrel{eq. \eqref{lemma0:eq4}}{=}-\left(-\sign t |t|^{2H} +\sign{s} |s|^{2H} -\sign(t-s)|t-s|^{2H}\right)\Gamma(-2H) \sin( \pi H).
\end{align}
Let now $H\in (\frac12,1)$ and $t,s>0.$ The integral in the left hand side of \eqref{lemma1:eq} is finite  because
$$\int_{1}^{+\infty}\left|\frac{\left(e^{i t y}-1\right)\left(e^{-i s y}-1\right)}{y^{2H+1}}\right|dy \leq 4 \int_{1}^{+\infty}\frac{1}{y^{2H+1}}dy <\infty$$ and
$$\int_{0}^{1}\left|\frac{\left(e^{i t y}-1\right)}{i y}\frac{\left(e^{-i s y}-1\right)}{-is y}\frac{1}{y^{2H-1}}\right|dy \leq |ts| \int_{0}^{1}\frac{1}{y^{2H-1}}dy <\infty.$$
That is why we can apply Fubini's theorem and rewrite the left hand side of \eqref{lemma1:eq} as
\begin{align}
\notag\int_{\R_+}\left(\int_{0}^t e^{i uy}du\right) \left(\int_{0}^s e^{-iv y}dv\right) \frac{1}{y^{2H-1}}dy&=\int_{\R_+}\int_{0}^t \int_{0}^s  \frac{e^{i (u-v)y}}{y^{2H-1}}du dv dy\\
\label{lemma:eq4}=&\int_{0}^t \int_{0}^s \int_{\R_+} \frac{e^{i (u-v)y}}{y^{2H-1}}dy du dv.
\end{align}
The imaginary part of \eqref{lemma:eq4} equals
\begin{align}
\notag\int_{0}^t \int_{0}^s \int_{\R_+} \frac{ \sin( (u-v)y)}{y^{2H-1}}&dy du dv =\int_{0}^t \int_{0}^s \sign(u-v) |u-v|^{2H-2} du dv \int_{\R_+} \frac{ \sin( z)}{z^{2H-1}}dz\\
\notag&=\frac{t^{2H}-s^{2H}+\sign{(t-s)} |t-s|^{2H}}{2H(2H-1)}\Gamma(2-2H)\sin (\pi H)\\
\label{lemma1:eq5}&=-\left(-t^{2H}+s^{2H}-\sign{(t-s)} |t-s|^{2H}\right)\Gamma(-2H)\sin (\pi H).
\end{align}
The other cases of $t,s$ are considered analogously.
We see that formulas \eqref{lemma1:eq5} and \eqref{lemma1:eq6} are the same. Therefore, they both are valid  for $H\in (0,\frac{1}{2})\cup (\frac12,1).$

To complete the proof, we note that
\begin{align*}
-\Gamma(-2H)&=\frac{\Gamma(1-2H)}{2H}=\frac{1}{2H \Gamma(2H)}\Gamma(1-2H)\Gamma(2H)\\
&=\frac{\pi}{2H \Gamma(2H)\sin (2\pi H)}=\frac{\pi}{\Gamma(1+2H)\sin (2\pi H)}.
\end{align*}

\end{proof}
\begin{remark}
\label{rem1}
In the case $s=t=1,$ formula \eqref{lemma1:eq} becomes
\begin{align}
\nonumber
\int_{\R_+}\frac{\left|e^{i y}-1\right|^2}{y^{2H+1}}dy&=\frac{\pi}{\Gamma(1+2H)\sin (2\pi H)} \left(e^{-i \pi H } + e^{i \pi H }\right)\\
\label{rem1:eq}
&=\frac{\pi}{\Gamma(1+2H)\sin (\pi H)}.
\end{align}

\end{remark}

\begin{lemma}
\label{lemma2}
Let $s,t\in \R,$ then
\begin{align}
\label{lemma2:eq1}
\Re \int_{\R_+}\frac{\left(e^{i t y}-1\right)\left(e^{-i s y}-1\right)}{y^{2}}dy&=
\frac{\pi}{2}(|t|+|s|-|t-s|).\\
\label{lemma2:eq2}
\Im \int_{\R_+}\frac{\left(e^{i t y}-1\right)\left(e^{-i s y}-1\right)}{y^{2}}dy
&=t \log |t|-s \log |s|- (t-s) \log |t-s|.
\end{align}
\end{lemma}
\begin{proof}
We can repeat the steps in \eqref{lemma1:eq7}.  Since $\int_{\R_+}(\cos z -1){z^{-2}}dy=\frac{\pi}{2},$  relation \eqref{lemma2:eq1} follows from \eqref{lemma1:eq7}, when $H=\frac12.$

Let $t>s>0,$ other cases are considered similarly.
The left hand side of \eqref{lemma2:eq2} equals
\begin{align*}
    &\int_{\R_+}\frac{\sin(y (t-s))+\sin (s y)-\sin (t y)}{y^2}dy=\int_{\R_+}\left[\sin((t-s) y)-(t-s)y\mathbb{I}\{(t-s)y\leq1\}\right.\\
    &\left.+\sin(s y)-sy\mathbb{I}\{sy\leq1\}-\sin (t y)+ty\mathbb{I}\{ty\leq1\}\right]{y^{-2}}dy\\
    &+\int_{\R_+}[-(t-s)y\mathbb{I}\{(t-s)y>1\}-sy\mathbb{I}\{sy>1\}+ty\mathbb{I}\{ty>1\}]{y^{-2}}dy=:I_1+I_2.
\end{align*}
By linearity,
\begin{align*}
    &\int_{0}^{+\infty}\left[\sin((t-s) y)-(t-s)y\mathbb{I}\{-(t-s)y\leq1\}\right]y^{-2}dy\\
    &=(t-s)\int_{0}^{+\infty}\left[\sin(z)-z\mathbb{I}\{z\leq1\}\right]z^{-2}dz<\infty,
\end{align*} therefore $I_1=0.$
For the second integral we have
\begin{align*}
I_2&=\lim_{N\to \infty}\left(-(t-s)\int_{1/(t-s)}^N \frac{dy}{y}-s\int_{1/s}^N \frac{dy}{y}+t\int_{1/t}^N \frac{dy}{y}\right)\\
   &=\lim_{N\to \infty}\left(-(t-s)\log N -(t-s) \log(t-s) -s\log N -s \log s+t\log N +t \log t -s\right)\\
   &=t \log t-s \log s- (t-s) \log (t-s).
\end{align*}
Thus, we obtain \eqref{lemma2:eq2}.
\end{proof}
\begin{remark}
\label{rem2}
Formula \eqref{rem1:eq} is valid in the case $H=\frac1 2$ too.
\end{remark}

\begin{lemma}
\label{lemma3}
Let $\ve \in \{-1,1\},$ then for any $H\in (0,1),H\neq 1/2$ we have
\begin{align}
 \nonumber    \int_{\R_+}&\frac{e^{i (t-x) \ve y}-e^{-i x \ve y}}{i\ve y^{H+1/2}}dy=\left((t-x)_+^{H-1/2}-(-x)_+^{H-1/2}\right)\Gamma\left(\frac12-H\right)\exp\left(-\frac{i\pi \ve}{2}\left(H+\frac12\right)\right)\\
   \label{lemma3:eq1}&-\left((t-x)_-^{H-1/2}-(-x)_-^{H-1/2}\right)\Gamma\left(\frac12-H\right)\exp\left(\frac{i\pi \ve}{2}\left(H+\frac12\right)\right).
\end{align}
\end{lemma}
\begin{proof}
Let us consider the real part of the left hand side of \eqref{lemma3:eq1}
\begin{align*}
    &\int_0^{+\infty}\frac{\sin((t-x) \ve y)-\sin (- x \ve y)}{\ve y^{H+1/2}}dy\\
    &=\left(|t-x|^{H-1/2}\sign(t-x)-|x|^{H-1/2}\sign(-x)\right)\int_0^{+\infty}\frac{\sin(z)}{z^{H+1/2}}dz\\
    &\stackrel{eq. \eqref{lemma0:eq4}}{=}\left(|t-x|^{H-1/2}\sign(t-x)-|x|^{H-1/2}\sign(-x)\right)
    \Gamma\left(\frac12-H\right)\cos\left(\frac\pi4+\frac{\pi H}{2}\right).
\end{align*}
For the imaginary part, consider two cases $H\in (0,1/2)$ and $H\in(1/2,1).$ If $H\in (0,1/2)$ then  the imaginary part of the left hand side of \eqref{lemma3:eq1} equals
\begin{align*}
    &\int_0^{+\infty}\frac{\cos((t-x) \ve y)-\cos (- x \ve y)}{-\ve y^{H+1/2}}dy=-\ve\int_0^{+\infty}\frac{\cos(|t-x|y)-\cos (|x|y)}{y^{H+1/2}}dy\\
    &=-\ve \left(|t-x|^{H-1/2}-|x|^{H-1/2}\right)\int_0^{+\infty}\frac{\cos(z)}{z^{H+1/2}}dz\\
    &\stackrel{eq. \eqref{lemma0:eq3}}{=}-\ve \left(|t-x|^{H-1/2}-|x|^{H-1/2}\right)\Gamma\left(\frac12-H\right)\sin\left(\frac\pi4+\frac{\pi H}{2}\right).
\end{align*}
For the case $H\in (1/2,1)$ we similarly have 
\begin{align*}
    &\int_0^{+\infty}\frac{\cos((t-x) \ve y)-\cos (- x \ve y)}{-\ve y^{H+1/2}}dy=-\ve\int_0^{+\infty}\left(\frac{\cos(|t-x|y)-1}{y^{H+1/2}}dy-\frac{\cos(|x|y)-1}{y^{H+1/2}}\right)dy\\
    &=-\ve \left(|t-x|^{H-1/2}-|x|^{H-1/2}\right)\int_0^{+\infty}\frac{\cos(z)-1}{z^{H+1/2}}dz\\
    &\stackrel{eq. \eqref{lemma0:eq2}}{=}-\ve \left(|t-x|^{H-1/2}-|x|^{H-1/2}\right)\Gamma\left(\frac12-H\right)\sin\left(\frac\pi4+\frac{\pi H}{2}\right).
\end{align*}
Therefore, the left hand side of \eqref{lemma3:eq1} equals
\begin{align*}
    &\left((t-x)_+^{H-1/2}-(-x)_+^{H-1/2}\right)\Gamma\left(\frac12-H\right)\left( \cos\left(\frac{\pi H}{2}+\frac\pi4\right)-i\ve\sin\left(\frac{\pi H}{2}+\frac\pi4\right) \right)\\
    &+\left((t-x)_-^{H-1/2}-(-x)_-^{H-1/2}\right)\Gamma\left(\frac12-H\right)\left(-\cos\left(\frac{\pi H}{2}+\frac\pi4\right)- i \ve\sin\left(\frac{\pi H}{2}+\frac\pi4\right) \right)\\
    &=\left((t-x)_+^{H-1/2}-(-x)_+^{H-1/2}\right)\Gamma\left(\frac12-H\right)\exp\left(-\frac{i\pi \ve}{2}\left(H+\frac12\right)\right)\\
    &-\left((t-x)_-^{H-1/2}-(-x)_-^{H-1/2}\right)\Gamma\left(\frac12-H\right)\exp\left(\frac{i\pi \ve}{2}\left(H+\frac12\right)\right).
\end{align*}
\end{proof}
\begin{lemma}
\label{lemma4}
Let $\ve \in \{-1,1\},t>0,x\in \R$ then 
\begin{align}
\label{lemma4:eq1}
    \int_0^{+\infty}\frac{e^{i (t-x) \ve y}-e^{-i x \ve y}}{i \ve y}dy&=\pi\mathbb{I}_{[0,t]}(x) - i \ve (\log |t-x|-\log |x|).
\end{align}
\end{lemma}
\begin{proof}
The real part of the left hand side of \eqref{lemma4:eq1} is computed similarly to Lemma \ref{lemma3}
\begin{align*}
    &\int_0^{+\infty}\frac{\sin((t-x) \ve y)-\sin (- x \ve y)}{\ve y}dy=\left(\sign(t-x)+\sign(x)\right)\int_0^{+\infty}\frac{\sin(z)}{z}dz=\pi\mathbb{I}_{[0,t]}(x).
\end{align*}
For the imaginary part, we get from \eqref{lemma0:eq1} that 
\begin{align*}
    &\int_0^{+\infty}\frac{\cos((t-x) \ve y)-\cos (- x \ve y)}{-\ve y}dy=-\ve (\log |t-x|-\log |x|).
\end{align*}
\end{proof}


\begin{proof}[Proof of Corrolary \ref{moving}] It follows from Theorem \ref{thm:s7} that function $g$ has the form \eqref{thm:s7:eq}. Let us make auxiliary notations.
\begin{align*}
    &p_H(t,x):=(t-x)_+^{H-1/2}-(-x)_+^{H-1/2}, \quad t>0,x\in \R,\\
    &f_H(t,x):=(t-x)_-^{H-1/2}-(-x)_-^{H-1/2}, \quad t>0,x\in \R,\\
    &a_0:=\sqrt{K_{(1,1)}}\prod_{j=1}^2 \frac{\Gamma\left(\frac12-H_j\right)}{\sqrt{2\pi}}, \quad 
    a_1:=\sqrt{K_{(-1,1)}}\prod_{j=1}^2 \frac{\Gamma\left(\frac12-H_j\right)}{\sqrt{2\pi}}, \\
    &\alpha_0:=\varphi_{(1,1)}, \quad  \alpha_1:= \varphi_{(-1,1)}, \quad \beta_1:=\frac{\pi }{2}\left(H_1+\frac12\right)\quad \beta_2:=\frac{\pi }{2}\left(H_2+\frac12\right).
\end{align*}
Relation \eqref{K:def2} in terms of $a_0$ and $a_1$ has a form
\begin{align}
    \nonumber a_0^2+a_1^2&=\frac{\Gamma^2\left(\frac12-H_1\right)}{2\pi}\frac{\Gamma^2\left(\frac12-H_2\right)}{2\pi} (K_{(1,1)}+K_{(-1,1)})\\
    \label{constant1}&=\frac12\prod_{j=1}^2\Gamma^2\left(\frac12-H_j\right)\frac{\Gamma(1+2H_j)\sin (\pi H_j)}{2 \pi^2}.
\end{align}
Recall that $\Gamma(1/2-H)=\pi\left(\sin (\pi (H+1/2))\Gamma(H+1/2)\right)^{-1}$ 
(e.g. \cite[p. 256]{adam}). Then \eqref{constant1} rewrites
\begin{align}
    \nonumber a_0^2+a_1^2=\frac12\prod_{j=1}^2 \frac{ \sin (\pi H_j) }{2\sin^2 (\pi H_j+\pi/2)}\frac{\Gamma(1+2H_j)}{\Gamma^2(H_j+1/2)}&=\frac12\prod_{j=1}^2 \frac{ \sin (\pi H_j) }{2\cos^2 (\pi H_j)}\frac{\Gamma(1+2H_j)}{\Gamma^2(H_j+1/2)}\\
    \label{const2_2}&=\frac{1}{8}\frac{ c_2^2(H_1) c_2^2(H_2)}{ \sin^2(2\beta_1)\sin^2(2\beta_2)}.
\end{align}
Now function $g$ from \eqref{thm:s7:eq} in the new notation reads as
\begin{align*}
g(\mathbf{t},\mathbf{x})&=a_0e^{i\alpha_0}\left(p_{H_1}(t_1,x_1)e^{-i\beta_1}-f_{H_1}(t_1,x_1)e^{i\beta_1}\right)\left(p_{H_2}(t_1,x_1)e^{-i\beta_2}-f_{H_2}(t_1,x_1)e^{i\beta_2}\right)\\
&+a_1e^{i\alpha_1}\left(p_{H_1}(t_1,x_1)e^{i\beta_1}-f_{H_1}(t_1,x_1)e^{-i\beta_1}\right)\left(p_{H_2}(t_1,x_1)e^{-i\beta_2}-f_{H_2}(t_1,x_1)e^{i\beta_2}\right)\\
&+a_1e^{-i\alpha_1}\left(p_{H_1}(t_1,x_1)e^{-i\beta_1}-f_{H_1}(t_1,x_1)e^{i\beta_1}\right)\left(p_{H_2}(t_1,x_1)e^{i\beta_2}-f_{H_2}(t_1,x_1)e^{-i\beta_2}\right)\\
&+a_0e^{-i\alpha_0}\left(p_{H_1}(t_1,x_1)e^{i\beta_1}-f_{H_1}(t_1,x_1)e^{-i\beta_1}\right)\left(p_{H_2}(t_1,x_1)e^{i\beta_2}-f_{H_2}(t_1,x_1)e^{-i\beta_2}\right).
\end{align*}
We rewrite it in the following form.
\begin{align*}
&g(\mathbf{t},\mathbf{x})\\
&=p_{H_1}(t_1,x_1)p_{H_2}(t_2,x_2)\left(a_0e^{i(\alpha_0-\beta_1-\beta_2)}+a_1e^{i(\alpha_1+\beta_1-\beta_2)}+a_1e^{i(-\alpha_1-\beta_1+\beta_2)}+a_0e^{i(-\alpha_0+\beta_1+\beta_2)}\right)\\
&-f_{H_1}(t_1,x_1)p_{H_2}(t_2,x_2)\left(a_0e^{i(\alpha_0+\beta_1-\beta_2)}+a_1e^{i(\alpha_1-\beta_1-\beta_2)}+a_1e^{i(-\alpha_1+\beta_1+\beta_2)}+a_0e^{i(-\alpha_0-\beta_1+\beta_2)}\right)\\
&-p_{H_1}(t_1,x_1)f_{H_2}(t_2,x_2)\left(a_0e^{i(\alpha_0-\beta_1+\beta_2)}+a_1e^{i(\alpha_1+\beta_1+\beta_2)}+a_1e^{i(-\alpha_1-\beta_1-\beta_2)}+a_0e^{i(-\alpha_0+\beta_1-\beta_2)}\right)\\
&+f_{H_1}(t_1,x_1)f_{H_2}(t_2,x_2)\left(a_0e^{i(\alpha_0+\beta_1+\beta_2)}+a_1e^{i(\alpha_1-\beta_1+\beta_2)}+a_1e^{i(-\alpha_1+\beta_1-\beta_2)}+a_0e^{i(-\alpha_0-\beta_1-\beta_2)}\right)\\
&=:p_{H_1}(t_1,x_1)p_{H_2}(t_2,x_2)A_{00}-f_{H_1}(t_1,x_1)p_{H_2}(t_2,x_2)A_{10}\\
&-p_{H_1}(t_1,x_1)f_{H_2}(t_2,x_2)A_{01}+f_{H_1}(t_1,x_1)f_{H_2}(t_2,x_2)A_{11}.
\end{align*}
We find such values of $a_0,a_1,\alpha_0,\alpha_1$ that coeficients $A_{10}=A_{01}=0,$ i.e., 
\begin{align*}
&A_{10}=a_0e^{i(\alpha_0+\beta_1-\beta_2)}+a_1e^{i(\alpha_1-\beta_1-\beta_2)}+a_1e^{i(-\alpha_1+\beta_1+\beta_2)}+a_0e^{i(-\alpha_0-\beta_1+\beta_2)}=0,\\
&A_{01}=a_0e^{i(\alpha_0-\beta_1+\beta_2)}+a_1e^{i(\alpha_1+\beta_1+\beta_2)}+a_1e^{i(-\alpha_1-\beta_1-\beta_2)}+a_0e^{i(-\alpha_0+\beta_1-\beta_2)}=0,
\end{align*}
or equivalently
\begin{align}
\label{alpha1-0}&a_1\left(e^{i\alpha_1}+ e^{-i\alpha_1}e^{2i(\beta_1+\beta_2)}\right)=-a_0\left(e^{i(\alpha_0+2\beta_1)}+e^{i(-\alpha_0+2\beta_2)}\right),\\
\label{alpha2-0}&a_1\left(e^{i\alpha_1}+ e^{-i\alpha_1}e^{-2i(\beta_1+\beta_2)}\right)=-a_0\left(e^{i(\alpha_0-2\beta_1)}+e^{i(-\alpha_0-2\beta_2)}\right).
\end{align}
Assume further that  $H_1+H_2\neq 1.$ Then we get that  
\begin{align}
\label{alpha1}
a_1e^{-i\alpha_1}\sin(2(\beta_1+\beta_2))&=-a_0\left(e^{i\alpha_0}\sin(2\beta_1)-e^{-i\alpha_0}\sin(2\beta_2)\right),\\
a_1\label{alpha2}e^{i\alpha_1}\sin(2(\beta_1+\beta_2))&= -a_0\left(e^{-i\alpha_0}\sin(2\beta_1)-e^{i\alpha_0}\sin(2\beta_2)\right),
\end{align}
and
\begin{equation}
\label{a1}
a_1^2=\frac{\sin^2(2\beta_1)+\sin^2(2\beta_2)+2\cos(2\alpha_0)\sin(2\beta_1)\sin(2\beta_2)}{\sin^2(2\beta_1+2\beta_2)}a_0^2.
\end{equation}

Under relation \eqref{alpha1},\eqref{alpha2} coefficient $A_{11}$ equals 
\begin{align}
\nonumber &A_{11}(\alpha_0)=a_0e^{i(\alpha_0+\beta_1+\beta_2)}-a_0  \frac{e^{-i\alpha_0}\sin(2\beta_1)+e^{i\alpha_0}\sin(2\beta_2)}{\sin(2(\beta_1+\beta_2))}e^{i(-\beta_1+\beta_2)}\\
\nonumber &+a_0e^{-i(\alpha_0+\beta_1+\beta_2)}-a_0\frac{e^{i\alpha_0}\sin(2\beta_1)+e^{-i\alpha_0}\sin(2\beta_2)}{\sin(2(\beta_1+\beta_2))}e^{i(\beta_1-\beta_2)}\\
\nonumber &=2 a_0\cos(\alpha_0+\beta_1+\beta_2)-2 a_0  \frac{\cos(\alpha_0+\beta_1-\beta_2)\sin(2\beta_1)+\cos(\alpha_0-\beta_1+\beta_2)\sin(2\beta_2)}{\sin(2\beta_1+2\beta_2)}\\
\nonumber &=2 a_0\cos(\alpha_0+\beta_1+\beta_2)\\
\nonumber &-2 a_0  \frac{\cos(\alpha_0+\beta_1+\beta_2)\cos(2\beta_2)+\sin(\alpha_0+\beta_1+\beta_2)\sin(2\beta_2)}{\sin(2\beta_1+2\beta_2)}\sin(2\beta_1)\\
\nonumber &-2 a_0  \frac{\cos(\alpha_0+\beta_1+\beta_2)\cos(2\beta_1)+\sin(\alpha_0+\beta_1+\beta_2)\sin(2\beta_1)}{\sin(2\beta_1+2\beta_2)}\sin(2\beta_2)\\
\label{a11}&=-4 a_0\frac{\sin(2\beta_1)\sin(2\beta_2)}{\sin(2\beta_1+2\beta_2)}\sin(\alpha_0+\beta_1+\beta_2).
\end{align}
Similarly, we get the value of $A_{00}(\alpha_0).$
\begin{align}
\nonumber &A_{00}(\alpha_0)=a_0e^{i(\alpha_0-\beta_1-\beta_2)}-a_0  \frac{e^{-i\alpha_0}\sin(2\beta_1)+e^{i\alpha_0}\sin(2\beta_2)}{\sin(2(\beta_1+\beta_2))}e^{i(\beta_1-\beta_2)}\\
\nonumber  &+a_0e^{i(-\alpha_0+\beta_1+\beta_2)}-a_0\frac{e^{i\alpha_0}\sin(2\beta_1)+e^{-i\alpha_0}\sin(2\beta_2)}{\sin(2(\beta_1+\beta_2))}e^{i(-\beta_1+\beta_2)}\\
\nonumber &=2 a_0\cos(\beta_1+\beta_2-\alpha_0)-2 a_0  \frac{\cos(-\alpha_0+\beta_1-\beta_2)\sin(2\beta_1)+\cos(-\alpha_0-\beta_1+\beta_2)\sin(2\beta_2)}{\sin(2\beta_1+2\beta_2)}\\
\label{a00}&=-4 a_0\frac{\sin(2\beta_1)\sin(2\beta_2)}{\sin(2\beta_1+2\beta_2)}\sin(-\alpha_0+\beta_1+\beta_2).
\end{align}

Now we want to find such $\alpha_0$ that function $g$ depends on ``the past only'', i.e., $A_{11}(\alpha_0)=0.$ From \eqref{a11} we get that $A_{11}(\tilde{\alpha}_0)=0$ for  $\tilde{\alpha}_0=-\beta_1-\beta_2=-\frac{1}{2}\pi(H_1+H_2)-\pi/2.$  

Denote for arbitrary $\alpha_0\in [0,2\pi]$ $\delta:=\alpha_0-\tilde{\alpha}_0.$ Then we rewrite \eqref{a00}, \eqref{a11} and \eqref{a1}
\begin{align}
\label{a00:1}A_{00}(\delta+\tilde{\alpha}_0)&=-4 a_0\frac{\sin(2\beta_1)\sin(2\beta_2)}{\sin(2\beta_1+2\beta_2)}\sin(\delta),\\
\nonumber A_{11}(\delta+\tilde{\alpha}_0)&=-4 a_0\frac{\sin(2\beta_1)\sin(2\beta_2)}{\sin(2\beta_1+2\beta_2)}\sin(2\beta_1+2\beta_2-\delta)\\
\label{a11:1}&=-4 \sin(2\beta_1)\sin(2\beta_2)\left(a_0 \cos (\delta) - \frac{\cos(2\beta_1+2\beta_2)}{\sin(2\beta_1+2\beta_2)}a_0\sin (\delta)\right), \\
\nonumber a_1^2&=a_0^2+4a_0^2\frac{\sin(2\beta_1)\sin(2\beta_2)}{\sin^2(2\beta_1+2\beta_2)}\sin(\delta)\sin(2\beta_1+2\beta_2-\delta),\\
\label{a1:1}&=a_0^2+\frac{A_{00}(\delta+\tilde{\alpha}_0)A_{11}(\delta+\tilde{\alpha}_0)}{4\sin(2\beta_1)\sin(2\beta_2)}.
\end{align}
From equations \eqref{a00:1} and \eqref{a11:1} we have 
\begin{align*}
    a_0^2&=a_0^2\sin^2 (\delta)+a_0^2\cos^2 (\delta)\\
    &=\frac{A^2_{00}(\delta+\tilde{\alpha}_0)+2A_{00}(\delta+\tilde{\alpha}_0)A_{11}(\delta+\tilde{\alpha}_0)\cos(2\beta_1+2\beta_2)+A^2_{11}(\delta+\tilde{\alpha}_0)}{4^2\sin^2(2\beta_1)\sin^2(2\beta_2)}.
\end{align*}
Combining  the last relation with \eqref{const2_2} and \eqref{a1:1}, we get 
\begin{align*}
    &\frac{ c_2^2(H_1) c_2^2(H_2)}{ \sin^2(2\beta_1)\sin^2(2\beta_2)}=8a_0^2+8a_1^2=4^2a_0^2+\frac{2A_{00}(\delta+\tilde{\alpha}_0)A_{11}(\delta+\tilde{\alpha}_0)}{\sin(2\beta_1)\sin(2\beta_2)}\\
    &=\frac{A^2_{00}(\delta+\tilde{\alpha}_0)+A^2_{11}(\delta+\tilde{\alpha}_0)}{\sin^2(2\beta_1)\sin^2(2\beta_2)}\\
    &+\frac{2A_{00}(\delta+\tilde{\alpha}_0)A_{11}(\delta+\tilde{\alpha}_0)}{\sin^2(2\beta_1)\sin^2(2\beta_2)}\left(\cos(2\beta_1+2\beta_2)+\sin(2\beta_1)\sin(2\beta_2)\right).
\end{align*}
Thus, we obtain
\begin{align*}
A^2_{00}(\delta+\tilde{\alpha}_0)+A^2_{11}(\delta+\tilde{\alpha}_0)+2A_{00}(\delta+\tilde{\alpha}_0)A_{11}(\delta+\tilde{\alpha}_0)\sin(\pi H_1)\sin(\pi H_2)=c_2^2(H_1)c_2^2(H_2)
\end{align*}
and
$$d_0^2+2d_0d_1\sin(\pi H_1)\sin(\pi H_2)+d_1^2=1.$$
Therefore, if $A_{11}(\delta+\tilde{\alpha}_0)=0$ then $\delta=0,$ $d_1=0$ and consequently $|d_0|=1.$

Consider now the case $H_1+H_2=1,$ which corresponds to $\beta_1+\beta_2=\pi.$ We get from \eqref{alpha1-0} and \eqref{alpha2-0} that 
\begin{align}
\label{alpha1-1}
&a_1\left(e^{i\alpha_1}+ e^{-i\alpha_1}\right)=-a_0\left(e^{i(\alpha_0+2\beta_1)}+e^{i(-\alpha_0+2\beta_2)}\right)=-a_0\left(e^{i(\alpha_0+2\beta_1)}+e^{i(-\alpha_0- 2\beta_1)}\right)\\
\label{alpha2-1}&a_1\left(e^{i\alpha_1}+ e^{-i\alpha_1}\right)=-a_0\left(e^{i(\alpha_0-2\beta_1)}+e^{i(-\alpha_0-2\beta_2)}\right)=-a_0\left(e^{i(\alpha_0-2\beta_1)}+e^{i(-\alpha_0+2\beta_1)}\right).
\end{align}

The solution has the form $\alpha_0=\pi k$ and $a_1 \cos (\alpha_1)=(-1)^{k+1}a_0 \cos(2\beta_1).$ Under relation \eqref{alpha1-1}, \eqref{alpha2-1} coefficient $A_{11}$ equals 
\begin{align}
\nonumber &A_{11}(\alpha_1)=a_0e^{i(\pi k+\pi)}+a_1e^{i(\alpha_1-\beta_1+\beta_2)}+a_1e^{i(-\alpha_1+\beta_1-\beta_2)}+a_0e^{i(-\pi k-\pi)}\\
\nonumber &=2a_0 (-1)^{k+1}+ 2a_1\cos(\alpha_1-\beta_1+\beta_2)=2a_0 (-1)^{k+1}- 2a_1\cos(\alpha_1-2\beta_1)\\
\label{a11-1}&=2a_1\left( \frac{\cos(\alpha_1)}{\cos(2\beta_1)}-\cos(\alpha_1)\cos(2\beta_1)-\sin(\alpha_1)\sin(2\beta_1)\right).
\end{align}

Similarly, we get the values of $A_{00}(\alpha_1)$ 
\begin{align}
\nonumber A_{00}(\alpha_1)&=a_0e^{i(\pi k-\pi)}+a_1e^{i(\alpha_1+\beta_1-\beta_2)}+a_1e^{i(-\alpha_1-\beta_1+\beta_2)}+a_0e^{i(-\pi k+\pi)}\\
\nonumber&= 2a_0 (-1)^{k+1}-2 a_1\cos(\alpha_1+2\beta_1)\\
\label{a00-1}&=2a_1\left( \frac{\cos(\alpha_1)}{\cos(2\beta_1)}- \cos(\alpha_1)\cos(2\beta_1)+ \sin(\alpha_1)\sin(2\beta_1)\right).
\end{align}

From equations \eqref{a11-1} and \eqref{a00-1} we have 
\begin{align*}
    4 a_1\sin (\alpha_1)\sin (2\beta_1)&=A_{00}(\alpha_1)-A_{11}(\alpha_1),\\
    4 a_1 \cos (\alpha_1)\frac{\sin^2 (2\beta_1)}{\cos (2\beta_1)}&=A_{00}(\alpha_1)+A_{11}(\alpha_1).
\end{align*}
Then relation $a_1^2\cos^2(a_1)=a_0^2\cos^2(2\beta_1)$ is equivalent to
$$(A_{00}(\alpha_1)+A_{11}(\alpha_1))^2\frac{\cos^2 (2\beta_1)}{4^2\sin^4 (2\beta_1)}=a_0^2\cos^2(2\beta_1).$$
Therefore, we get
\begin{align}
 \nonumber   a_1^2+a_0^2&=\frac{1}{4^2}\frac{(A_{00}(\alpha_1)-A_{11}(\alpha_1))^2}{\sin^2 (2\beta_1)}+\frac{1}{4^2}\frac{(A_{00}(\alpha_1)+A_{11}(\alpha_1))^2}{\sin^4 (2\beta_1)}\cos^2 (2\beta_1)\\
 \nonumber   &+\frac{1}{4^2}\frac{(A_{00}(\alpha_1)+A_{11}(\alpha_1))^2}{\sin^4 (2\beta_1)}\\
\label{fl}    &=\frac{1}{8 \sin^4(2\beta_1)}\left(A^2_{00}(\alpha_1)+2A_{00}(\alpha_1)A_{11}(\alpha_1)\cos^2 (2\beta_1) + A^2_{00}(\alpha_1)\right).
\end{align}
Since \eqref{constant1} and $\sin(2\beta_1)=\cos(2\beta_2)$ in the case $\beta_1+\beta_2=\pi,$  we see that relation \eqref{fl} is equivalent to \eqref{dd}.


\end{proof}


\end{document}